\documentclass[11pt]{article}
\usepackage{layout}
\usepackage[makeroom]{cancel}
\usepackage{bbm}

\usepackage{amsfonts,dsfont,amsmath,amsthm}
\usepackage{amssymb,bm}
\usepackage{graphicx} 
\usepackage{mathtools}
\usepackage{multicol}
\usepackage{mathrsfs}
\usepackage{enumerate,enumitem}
\usepackage[rgb,dvipsnames]{xcolor}
\usepackage{float}
\usepackage[normalem]{ulem}
\usepackage{circuitikz}
\usepackage{cite}
\usepackage[colorlinks,linkcolor=blue,citecolor=blue,pagebackref,hypertexnames=false, breaklinks]{hyperref}
\usepackage{float}
\allowdisplaybreaks 

\newtheorem{theorem}{Theorem}[section]
\newtheorem{assumption}[theorem]{Assumption}
\newtheorem{definition}[theorem]{Definition}
\newtheorem{proposition}[theorem]{Proposition}
\newtheorem{prop}[theorem]{Proposition}
\newtheorem{corollary}[theorem]{Corollary}
\newtheorem{lemma}[theorem]{Lemma}
\newtheorem{remark}[theorem]{Remark}
\newtheorem{example}[theorem]{Example}
\newtheorem{examples}[theorem]{Examples}

\newtheorem{open question}[theorem]{Open Question}
\newtheorem{c/p}[theorem]{Conjecture/Proposition}

\def\Xint#1{\mathchoice
    {\XXint\displaystyle\textstyle{#1}}%
    {\XXint\textstyle\scriptstyle{#1}}%
    {\XXint\scriptstyle\scriptscriptstyle{#1}}%
    {\XXint\scriptscriptstyle\scriptscriptstyle{#1}}%
    \!\int}
    \def\XXint#1#2#3{{\setbox0=\hbox{$#1{#2#3}{\int}$}
    \vcenter{\hbox{$#2#3$}}\kern-.5\wd0}}
    \def\fint{\Xint-}
    
\def\vint{\mathop{\mathchoice%
 {\setbox0\hbox{$\displaystyle\intop$}\kern 0.22\wd0%
 \vcenter{\hrule width 0.6\wd0}\kern -0.82\wd0}%
 {\setbox0\hbox{$\textstyle\intop$}\kern 0.2\wd0%
 \vcenter{\hrule width 0.6\wd0}\kern -0.8\wd0}%
 {\setbox0\hbox{$\scriptstyle\intop$}\kern 0.2\wd0%
 \vcenter{\hrule width 0.6\wd0}\kern -0.8\wd0}%
 {\setbox0\hbox{$\scriptscriptstyle\intop$}\kern 0.2\wd0%
 \vcenter{\hrule width 0.6\wd0}\kern -0.8\wd0}}%
 \mathopen{}\int}

\newcommand{\ve}{\varepsilon}

\mathtoolsset{showonlyrefs=true}

\usepackage[textwidth=3.2cm]{todonotes}

\title{Orlicz-Sobolev embeddings \\ and heat kernel based Besov classes}

\author{Patricia Alonso Ruiz, Fabrice Baudoin\footnote{F.B.'s research is partially funded by Villum Fonden through a Villum Investigator grant and by grant 10.46540/4283-00175B from Independent Research Fund Denmark}}

\date{\today}

\begin{document}

\maketitle
\begin{abstract}
    This paper investigates functional inequalities involving Besov spaces and functions of bounded variation, when the underlying metric measure space displays different local and global structures. Particular focus is put on the $L^1$ theory and its applications to sets of finite perimeter and isoperimetric inequalities, which can now capture such structural differences.
\end{abstract}

\tableofcontents

\newpage

\section{Introduction}

There are instances, when the local and global structure of a space differ, affecting the space-time scaling behavior of its intrinsic diffusion process: instead of being just of power type, the scaling becomes of ``variable'' power type.  
Such spaces commonly appear in research on fractals, see e.g.~\cite{BBK06,GY24} and references therein. 
Concerning research in functional inequalities, in particular those involving Besov spaces and functions of bounded variation, previous work 
in the metric measure space  setting~\cite{ABCRST1,ABCRST3,ARB21,gao2025heatkernelbasedpenergynorms}, was not able to capture differences in local and global behaviors. The chief reason for that shortcoming was that function spaces and inequalities were restricted to using standard $L^p$-norms, which could not take on account variable exponents. In the Euclidean context, Besov spaces with variable exponents have recently been investigated in~\cite{Dri19,DZ23}. 

\medskip

The main characteristic features of the metric measure spaces $(X,d,\mu)$ motivating this project are spelled out in terms of the heat semigroup $\{P_t\}_{t\geq 0}$ associated with their intrinsic diffusion: First, the semigroup admits a heat kernel $p_t(x,y)$ satisfying upper and lower estimates that change depending on the time-space scaling. More precisely, 
\begin{equation}\label{E:HK_estimates_intro}
p_t(x,y) \asymp \frac{1}{\sigma_{\alpha_1/\beta_1,\alpha_2/\beta_2 }(t)} \exp\bigg(\!\!-c t \sigma_{\beta_2/(\beta_2-1),\beta_1/(\beta_1-1) } \bigg( \frac{d(x,y)}{t}\bigg)\bigg)
\end{equation}
for some $0<\alpha_i<\beta_i$ with $2\leq \beta_i\leq 1+\alpha_i$ and $\mu$-a.e. $x,y\in X$, where
\[
\sigma_{\nu_1,\nu_2} (r):=r^{\nu_1} 1_{[0,1]}(r) +r^{\nu_2} 1_{(1,+\infty)}(r),\qquad r>0.
\]
The expression $A\asymp B$ is meant to indicate that $cA\leq B\leq CA$ for some constants $C,c>0$ which are irrelevant. 

\medskip

A second property has a more geometric flavor and is a generalization of what is generally referred to as \emph{weak Bakry-\'Emery} condition, see e.g.~\cite{ABCRST3,GY24}. Specifically, the heat semigroup $\{P_t\}_{t\geq 0}$ is said to satisfy the property ${\rm wBE}(\kappa_1,\kappa_2)$ if the H\"older estimate
\begin{equation}\label{E:wBE_intro}
|P_tf(x)-P_tf(y)|\leq C_{\rm wBE}\frac{\sigma_{\kappa_1,\kappa_2}\big(d(x,y)\big)}{\sigma_{\frac{\kappa_1}{\beta_1},\frac{\kappa_2}{\beta_2}}(t)}\|f\|_{L^\infty}
\end{equation}
holds for some suitable indices $\kappa_1,\kappa_2$, and $\beta_1,\beta_2$ from~\eqref{E:HK_estimates_intro}. Condition~\eqref{E:wBE_intro} in particular renders the space a certain \emph{weak monotonicity} property, c.f. Assumption~\ref{A:PPI}, which will be key to obtain the main results of the paper. This condition has lately played a key role in the investigation of first-order Sobolev spaces and $p$-energies on fractals, see e.g.~\cite{Bau24,CGYZ24,KS24} and references therein.

\medskip

Back to the question concerning the ``correct'' function space to capture the different structure of the underlying fractal, we will see how Orlicz spaces provide the adequate replacement. Orlicz spaces are a generalization of $L^p$ spaces and thus natural candidates to study  embeddings in domains with properties involving different power types. Sobolev embeddings into Orlicz spaces on Euclidean domains and manifolds have been investigated in~\cite{CM23}. 

\medskip

Besides Orlicz-Sobolev inequalities, the present paper connects the latter with the semigroup based Besov spaces $\mathbf{B}^{h,p}(X)$ introduced in~\cite{ABCRST1} and their characterization as Korevaar-Schoen energy functionals 
\[
E_{p,\Psi}(f,r)=\frac{1}{\Psi(r)^p}\int_X\fint_{B(y,r)}|f(x)-f(y)|^pd\mu(y)\,d\mu(x),
\]
where $\Psi\colon [0,\infty)\to[0,\infty)$ is and increasing homeomorphism such that $\Psi(0)=0$, and $B(y,r)$ denotes a ball in $X$ of radius $r>0$ centered at $y\in X$. These functionals have been subject of extensive work in the fractal setting, see e.g.~\cite{ABCRST3,KS24,MS25,ARB25}. Of particular interest in the present paper are results in the case $p=1$, which has not yet been treated in related investigations~\cite{KS24,MS25}.

\medskip

The paper is organized as follows: Section~\ref{S:hBesov_spaces} introduces generalized heat semigroup based Besov spaces and proves in Theorem~\ref{T:pol} the relevant Orlicz-Sobolev inequality. Section~\ref{S:KS_char} deals with a Korevaar-Schoen type characterization of the aforementioned Besov spaces that will allow to show the continuity of the semigroup and a useful pseudo-Poincar\'e inequality, c.f. Theorem~\ref{T:continuity} and Proposition~\ref{P:pseudoPI} respectively, under the weak monotonicity assumption. Section~\ref{S:case_p_1} shows how the weak Bakry-\'Emery condition~\eqref{E:wBE_intro} implies the latter weak monotonicity when $p=1$, see Theorem~\ref{T:weak_monotonicity}.

\section{Generalized heat semigroup based Besov classes}\label{S:hBesov_spaces}
Throughout this section, the underlying space $(X,\mu)$ is a measurable space equipped with a $\sigma$-finite measure $\mu$, for which the measure decomposition theorem holds and for which there exists a countable family generating the $\sigma$-algebra of $X$, see e.g.~\cite[page 7]{BGL} for a description of this framework. The space of Borel functions on $X$ will be denoted by $\mathcal{B}(X)$, and the space of non-negative Borel functions by $\mathcal{B}^+(X)$.

\medskip

Further, let $(\mathcal{E},\mathcal{F}=\mathbf{dom}(\mathcal{E}))$ be a Dirichlet form on $L^2(X,\mu)$, and let $\{P_{t}\}_{t\geq 0}$ denote the associated (heat) semigroup on $L^2(X,\mu)$ associated with the Dirichlet space $(X,\mu,\mathcal{E},\mathcal{F})$. 
It will always be assumed that the semigroup $\{P_{t}\}_{t\geq 0}$ is conservative, i.e. that $P_t 1=1$. Due to this assumption, the semigroup yields a family of heat kernel measures, c.f.~\cite[Theorem 1.2.3]{BGL}. In particular, for every bounded $f\in\mathcal{B}(X)$ or $f\in\mathcal{B}(X)$, 
\begin{align}\label{heat kernel measure}
P_tf (x)=\int_X f(y) p_t(x,dy)
\end{align}
for any $t\geq 0$ and $x\in X$, where $p_t(x,dy)$ is a probability kernel for each $t>0$. 

\subsection{Preliminaries: Orlicz spaces}
Since one of the main tools in the paper are Orlicz spaces, we start by recalling their definition and some basic results, referring the reader to the textbook~\cite{RR91} for further details. 

\begin{definition}\label{D:Young_fct}
Let $\Phi\colon[0,+\infty)\to[0,+\infty]$ be a right-continuous, positive, non-decreasing function satisfying $\Phi(0)=0$ and $\lim\limits_{\tau\to\infty}\Phi(\tau)=+\infty$.
A function $\phi\colon[0,+\infty)\to[0,+\infty]$ that admits the representation
\begin{equation}\label{E:Young_rep}
\phi(x)=\int_0^x\Phi(\tau)\,d\tau
\end{equation}
is called a Young function. The function given by the Fenchel transform
\begin{equation}\label{E:def_Young_conjugate}
\psi(t):=\sup_{s>0}\{st-\phi(s)\}    
\end{equation}
is called the Young conjugate of $\phi$.
\end{definition}
From the representation~\eqref{E:Young_rep} one can infer that a Young function is convex, continuous, increasing and $\lim\limits_{x\to 0}\phi(x)=0$. In addition, any pair of Young conjugate functions $\phi,\psi$ always satisfies
\begin{equation*}\label{E:Young_conj}
s\leq\phi^{-1}(s)\psi^{-1}(s)\leq 2s\qquad\forall\,s>0.
\end{equation*}


In this exposition, the Orlicz space associated with $\phi$ is defined by duality.

\begin{definition}\label{D:Orlicz_sp}
Let $\phi$ be a Young function with Young conjugate $\psi$. The Orlicz space is defined as
\begin{equation*}\label{E:Orilcz_sp}
L^{\phi}(X,\mu):=\Big\{f\in\mathcal{B}(X)\,\colon\,\int_X |f| g  \,d\mu<\infty\text{ for all }g\in\mathcal{B}^+(X)\,\text{ with }\,\int_X\psi(g)\,d\mu<\infty\Big\},
\end{equation*}
and its corresponding norm as
\begin{equation}\label{Orlicz definition}
\|f\|_{\phi}:=\sup\Big\{\int_X|fg|\,d\mu\,\colon\, g\in \mathcal{B}(X), \, \int_X\psi(|g|)\,d\mu\leq 1\Big\}.
\end{equation}
\end{definition}
In particular, if $\phi(x)=x^p$ for some $p \ge 1$, then $L^\phi(X,\mu)=L^p(X,\mu)$.
If the function $\phi$ is doubling, i.e. there exists $B>0$ such that
\begin{equation}\label{E:Cond_Delta2}
\phi(2s)\leq B\phi(s)\qquad\forall\,s>0,
\end{equation}
then the norm $\|f\|_\phi$ is equivalent to the Luxembourg norm
\begin{equation}\label{E:Luxembourg}
\inf\Big\{s>0\,\colon\,\int_X\phi\bigg(\frac{|f|}{s}\bigg)\,d\mu\leq 1\Big\}.
\end{equation}

The Orlicz spaces involved in the results of the present paper will mainly involve doubling Young functions of the type
\begin{equation*}
    \phi(s)=s^\gamma \wedge s^{\kappa} \qquad \text{with}\qquad 1 \le \gamma \le \kappa.
\end{equation*}
We finish this paragraph by pointing out  the following known property of Orlicz norms, including the proof for the reader's convenience.

\begin{lemma}\label{P:Orlicz_properties}
Let $\phi$ be a Young function with Young conjugate $\psi$. For any Borel set $E\subset X$ with positive and finite measure,
        \begin{equation}\label{E:Norm_1E}
        \|\mathbf{1}_E\|_\phi=\mu(E)\psi^{-1} \bigg(\frac{1}{\mu(E)}\bigg).
        \end{equation}
\end{lemma}
\begin{proof}
First, applying~\eqref{Orlicz definition} with $g=\psi^{-1}\Big(\frac{1}{\mu(E)} \Big) \mathbf{1}_E$ yields
 \[
 \mu(E)\psi^{-1} \Big(\frac{1}{\mu(E)}\Big) \le  \|\mathbf{1}_E\|_\phi .
 \]
Let now $g\in \mathcal{B}^+(X), \, \int_X\psi(g)\,d\mu\leq 1$. From Jensen's inequality, it follows that
\begin{align*}
\int_E  g d\mu =\int_E \psi^{-1} (\psi (g)) d\mu \le \mu(E) \psi^{-1} \left( \frac{1}{\mu(E)} \int_X\psi(g)\,d\mu \right)\le \mu(E)\psi^{-1} \Big(\frac{1}{\mu(E)}\Big)
\end{align*}
and therefore $\mu(E)\psi^{-1} \Big(\frac{1}{\mu(E)}\Big) \ge  \|\mathbf{1}_E\|_\phi$.
\end{proof}





\subsection{Generalized heat semigroup based Besov classes}

Heat semigroup based Besov classes were introduced in \cite{ABCRST1,ABCRST3}, we generalize here their definitions. Throughout this section, we consider $1 \le p <+\infty$ and $h\colon[0,+\infty)\to [0,+\infty)$ a non-decreasing function such that $h(0)=0$. 

\begin{definition}\label{def:Besov}
The $h$-Besov seminorm is defined as
\begin{equation}\label{E:def_seminorm_ph}
\| f \|_{p,h}:= \sup_{t >0} \frac{1}{h(t)} \left( \int_X P_t (|f-f(y)|^p)(y) d\mu(y) \right)^{1/p}
\end{equation}
for any $f \in L^p(X,\mu)$.
\end{definition}

\begin{remark}
In terms of the heat kernel measure $p_t(x,dy)$ associated with the heat semigroup $\{P_t\}_{t\geq 0}$, the integral in~\eqref{E:def_seminorm_ph} reads $\displaystyle\int_X \int_X |f(x)-f(y)|^p p_t(x,dy)d\mu(y)$.   
\end{remark}

The purpose of the paper is to study functional inequalities involving the Besov type spaces
\begin{equation}\label{eq:def:Besov}
\mathbf{B}^{p,h}(X):=\{ f \in L^p(X,\mu)\, :\,  \| f \|_{p,h} <+\infty \}
\end{equation}
whose associated norm is defined as
\[
\|f \|_{\mathbf{B}^{p,h}(X)}:=\| f \|_{L^p(X,\mu)} + \| f \|_{p,h}.
\]
The same proof as~\cite[Proposition 4.14]{ABCRST1} works in this case to show:
\begin{prop}
The space $(\mathbf{B}^{p,h}(X), \| \cdot \|_{\mathbf{B}^{p,h}(X)})$ is a Banach space. 
\end{prop}

We also provide following alternative characterization of $\mathbf{B}^{p,h}(X)$.

\begin{lemma}\label{Lemma limsup debut}
\[
\mathbf{B}^{p,h}(X)
=\left\{ f \in L^p(X,\mu)\, :\,  \limsup_{t \to 0} \frac{1}{h(t)} \left( \int_X P_t (|f-f(y)|^p)(y) d\mu(y) \right)^{1/p} <+\infty \right\}.
\]
Furthermore, for any $f \in \mathbf{B}^{p,h}(X)$ and $t >0$, 
\begin{align}\label{equivalence sup-limsup}
\| f \|_{p,h} \le \frac{2}{h(t)} \| f \|_{L^p(X,\mu)} +\sup_{s\in (0,t]} \frac{1}{h(s)} \left( \int_X P_s (|f-f(y)|^p)(y) d\mu(y) \right)^{1/p}.
\end{align}
\end{lemma}

\begin{proof}
If $f \in \mathbf{B}^{p,h}(X)$, then we clearly have
\[
\limsup_{t \to 0} \frac{1}{h(t)} \left( \int_X P_t (|f-f(y)|^p)(y) d\mu(y) \right)^{1/p} \le \| f \|_{p,h}.
\]
Conversely, if $\limsup\limits_{t \to 0} \frac{1}{h(t)} \left( \int_X P_t (|f-f(y)|^p)(y) d\mu(y) \right)^{1/p} <+\infty$, then
there is some $\ve >0$ for which
\[
\sup_{t \in (0,\ve]} \frac{1}{h(t)} \left( \int_X P_t (|f-f(y)|^p)(y) d\mu(y) \right)^{1/p} <\infty.
\]
For $t >\ve$, applying $|f(x)-f(y)|^p \le 2^{p-1} (|f(x)|^p+|f(y)|^p)$, the fact that semigroup is conservative and the bound $\int_XP_t(|f|^p)(x)\, d\mu(x)\le \int_X|f(x)|^p\, d\mu(x)$,
we have
\begin{align*}
 \frac{1}{h(t)} \left( \int_X P_t (|f-f(y)|^p)(y) d\mu(y) \right)^{1/p} \le  \frac{2}{h(\ve)} \| f \|_{L^p(X,\mu)}.
\end{align*}
Therefore, 
\[
\sup_{t > 0} \frac{1}{h(t)} \left( \int_X P_t (|f-f(y)|^p)(y) d\mu(y) \right)^{1/p} <+\infty.
\]
and~\eqref{equivalence sup-limsup} follows.
\end{proof}

As a direct consequence of the latter,

\begin{corollary}
  Let $g:[0,+\infty) \to [0,+\infty)$ be a non decreasing function such that $g(0)=0$ and 
  \[
  a\, h(t) \le g(t) \le b\, h(t)
  \]
  for $t \in [0,t_0]$, where $a,b,t_0>0$. Then, $\mathbf{B}^{p,g}(X)=\mathbf{B}^{p,h}(X)$ with equivalent norms.
\end{corollary}

We finish this paragraph with a pseudo-Poincar\'e type inequality that will play an important role in the sequel.

\begin{lemma}\label{pseudo-poincare 1}
For every $f \in \mathbf{B}^{p,h}(X)$ and $t \ge 0$
\[
\| P_tf -f \|_{L^p(X,\mu)} \le h(t) \| f \|_{p,h}
\]
\end{lemma}

\begin{proof}
Since the semigroup is conservative, H\"older's inequality yields
\begin{align*}
 \left(\int_X | P_t f (x)-f(x)|^p d\mu(x)\right)^{1/p} & = \left(\int_X | P_t (f -f(x))(x)|^p d\mu(x)\right)^{1/p} \\
 & \le \left( \int_X P_t (|f-f(x)|^p)(x) d\mu(x) \right)^{1/p}
 \le h(t)\, \| f \|_{p,h}.
\end{align*}
\end{proof}

\subsection{Isoperimetric inequalities}
Throughout this section, we assume $p=1$. As in \cite{ABCRST1,ABCRST3}, the case $p=1$ has a special geometric interpretation related to the theory of BV functions, that will be further analyzed in Section~\ref{S:case_p_1}. In particular, we generalize to our setting the definitions of \cite{ABCRST1}.

\begin{definition}\label{D:h-perimeter}
    Let $E\subset X$ be a measurable set of finite $\mu$-measure. The $h$-perimeter of $E$ is defined as
    \begin{equation*}
        {\rm Per}_h(E):=\|\mathbf{1}_E\|_{1,h}
    \end{equation*}
   
\end{definition}

In the present  context, it is possible to obtain the following general isoperimetric inequality.
\begin{prop}\label{P:general_isoperimetric}
Assume that $h$ is increasing with $\lim_{t\to +\infty} h(t)=+\infty$ and that there exists a non-decreasing function $V\colon [0,\infty)\to [0,\infty)$ satisfying $\lim\limits_{t \to +\infty} V(t)=+\infty$ such that 
\begin{equation}\label{eq:Ptinfty_to_L1}
\| P_t f \|_{L^\infty(X,\mu)} \leq \frac{1}{V(t)}\|  f \|_{L^1(X,\mu)}
\end{equation}
for every $f \in L^1(X,\mu)$ and any $t> 0$.
Then, for any set $E\subset X$ of finite $h$-perimeter,
    \begin{equation}\label{E:general_isoperimetric}
        \mu(E)^2 \, \varphi  \left( \frac{2}{\mu(E)} \right) \le {\rm Per}_h(E) ,
    \end{equation}
    where $\varphi$ is the Fenchel transform
    \[
\varphi(t)= \sup_{s>0} \left( st-\frac{s}{V\left(h^{-1}(1/s) \right)} \right).
    \]
\end{prop}

\begin{proof}
We suitable adapt the proof of \cite[Proposition 6.5]{ABCRST1} which goes back to ideas by M. Ledoux~\cite{ledoux}. The first part of the argument is similar but we reproduce it for the sake of completeness. Since the semigroup is symmetric and conservative,
\begin{align*}
\Vert P_t \mathbf{1}_E -\mathbf 1_E  \Vert_{L^1(X,\mu)} =& \int_E (1-P_t \mathbf 1_E ) d\mu + \int_{E^c} P_t \mathbf 1_E\,  d\mu\\
 =& \int_E (1-P_t \mathbf 1_E ) d\mu + \int_E P_t \mathbf1_{E^c}\, d\mu\\
 =& 2 \left( \mu(E)- \int_E P_t \mathbf 1_E \, d\mu \right).
\end{align*}
Moreover, assumption~\eqref{eq:Ptinfty_to_L1} implies
\begin{equation*}
	\int_E P_t \mathbf 1_E \, d\mu \leq \frac{1}{V(t)} \mu(E)^2.
	\end{equation*}
Combining these equations with the pseudo-Poincar\'e inequality from Lemma~\ref{pseudo-poincare 1} we obtain, for $t>0$,
\begin{equation*}
2\mu(E) \leq  \|P_t \mathbf 1_E - \mathbf 1_E\|_{L^1(X,\mu)}+2 \int_E P_t \mathbf 1_E \, d\mu
	\leq   h(t) \ {\rm Per}_h(E) +\frac{2}{V(t)}   \mu(E)^2.
\end{equation*}
If ${\rm Per}_h(E)=0$, then letting $t\to\infty$ we also see that $\mu(E)=0$, and so without loss of generality we may
assume that ${\rm Per}_h(E)>0$, in which case we also have $\mu(E)>0$. Setting $u:=\frac{1}{h(t)}$, the previous inequality can be rewritten as
\[
\frac{2}{\mu(E)} u- \frac{u}{V\left(h^{-1}(1/u) \right)} \le \frac{{\rm Per}_h(E)}{\mu(E)^2}.
\]
Taking the supremum over $u>0$ on the left hand side finishes the proof.
\end{proof}

A second geometric quantity of interest is the $h$-Cheeger constant of $X$, which is defined as 
\begin{equation}\label{E:def_Cheeger_const}
    C_{h}:=\inf\Big\{\frac{{\rm Per}_h(E)}{\mu(E)}\colon E\subseteq X\text{ with }\mu(E)<1/2\text{ and }{\rm Per}_h(E)<\infty\Big\}
\end{equation}
assuming $\mu(X)=1$. 
When the Dirichlet form $(\mathcal{E},\mathcal{F})$ associated with the semigroup $\{P_t\}_{t\geq 0}$ admits the spectral inequality
\begin{equation}\label{E:spectral_ineq}
    \int_Xf^2d\mu-\bigg(\int_X f\,d\mu\bigg)^2\leq \frac{\mathcal{E}(f,f)}{\lambda_1},
\end{equation}
then it is possible to bound from below~\eqref{E:def_Cheeger_const} in a similar way as in~\cite[Theorem 7.1]{ABCRST1}. 

\begin{theorem}\label{T:Cheeger_const_bound}
    Assume $\mu(X)=1$. If the spectral inequality~\eqref{E:spectral_ineq} holds, then
    \begin{equation}\label{E:Cheeger_const_bound}
        C_h\geq \sup_{t>0}\frac{1-e^{-\lambda_1t}}{h(t)}.
    \end{equation}
\end{theorem}
\begin{proof}
    Note first that, since the semigroup is symmetric and conservative
    \begin{align*}
        \|P_t\mathbf{1}_E-\mathbf{1}_E\|_{L^1(X,\mu)}
        &=\int_E(1-P_t\mathbf{1}_E)\,d\mu+\int_EP_t\mathbf{1}_{X{\setminus}E}d\mu
        =2\Big(\mu(E)-\int_EP_t\mathbf{1}_E\,d\mu\Big)\\
        &=2\big(\mu(E)-\|P_{t/2}\mathbf{1}_E\|^2_{L^2(X,\mu)}\big).
    \end{align*}
    Applying Lemma~\ref{pseudo-poincare 1} we therefore get
    \begin{equation}\label{E:Cheeger_const_01}
        \mu(E)\leq \frac{h(t)}{2}{\rm Per}_h(E)+\|P_{t/2}\mathbf{1}_E\|^2_{L^2(X,\mu)}.
    \end{equation}
    Further, by virtue of the spectral theorem
    \begin{equation*}
      \|P_{t/2}\mathbf{1}_E\|^2_{L^2(X,\mu)}=\mu(E)^2+\|P_{t/2}(\mathbf{1}_E-\mu(E))\|_{L^2(X,\mu)}^2\leq \mu(E)^2+e^{-\lambda_1t}\|\mathbf{1}_E-\mu(E)\|_{L^2(X,\mu)}^2  
    \end{equation*}
    which together with~\eqref{E:Cheeger_const_01} yields 
    \begin{equation*}
        \frac{h(t)}{2}{\rm Per}_h(E)\geq \mu(E)(1-\mu(E))(1-e^{-\lambda_1t}).
    \end{equation*}
    Thus,
    \begin{equation*}
    \frac{{\rm Per}_h(E)}{\mu(E)}\geq \frac{2}{h(t)}(1-\mu(E))(1-e^{-\lambda_1t})
    \end{equation*}
    for all $t>0$ and $E\subset X$ with $\mu(E)\leq \frac{1}{2}$, and so~\eqref{E:Cheeger_const_bound} follows.
\end{proof}

\subsection{Orlicz-Sobolev inequalities}\label{S:Orlicz-Sobolev}
The goal of this section is to prove Orlicz-Sobolev inequalities using the general results of~\cite{BCLS95}. 
To do so we will require that the seminorm $\|\cdot\|_{p,h}$ satisfies the monotonicity property~\eqref{E:PPI}. In Section~\ref{S:case_p_1} we provide a condition that guarantees the latter with $p=1$  in the metric spaces mentioned in Remark~\ref{R:examples}.

\begin{assumption}\label{A:PPI}
There exists a constant $C>0$ such that for every $f \in \mathbf{B}^{p,h}(X)$
\begin{equation}\label{E:PPI}
\| f \|_{p,h}\le C \liminf_{t\to 0^+} \frac{1}{h(t)} \left( \int_X P_t (|f-f(y)|^p)(y) d\mu(y) \right)^{1/p}.
\end{equation}
\end{assumption}

\begin{theorem}\label{T:pol}
Assume that $h$ is increasing and that $\lim_{t\to +\infty}h(t)=+\infty$ and that there exists a non-decreasing function $V\colon [0,\infty)\to [0,\infty)$ satisfying $\lim\limits_{t \to +\infty} V(t)=+\infty$ such that for every $f \in L^1(X,\mu) $
\begin{equation}\label{eq:subGauss-upper4}
\| P_t f \|_{L^\infty(X,\mu)} \leq \frac{1}{V(t)}\|  f \|_{L^1(X,\mu)}
\end{equation}
for any $t> 0$. Assume that $\phi$ is a Young function satisfying the doubling condition~\eqref{E:Cond_Delta2} and such that for every $s >0$
\begin{equation}\label{E:ass_hV}
(h\circ V^{-1}) (s)^p \le s\phi^{-1}(1/s),
\end{equation}
where we denote $V^{-1}(t)=\inf \{ s : V(s) >t\}$.
Then, for every $f \in \mathbf{B}^{p,h}(X)$,
\begin{equation}\label{E:pol}
\| f^p \|_\phi^{1/p} \le C  \| f \|_{p,h}.
\end{equation}
\end{theorem}

\begin{proof}
We apply~\cite[Corollary 10.7]{BCLS95} with $r=p$ and $M_s=P_{h^{-1}(s)}$. Note however that there is a typo in the statement of~\cite[Corollary 10.7]{BCLS95}, and that $\phi(t) \le t^{p/r} \Phi^{-1}(1/t)$ should read $\phi(t)^p \le t^{p/r} \Phi^{-1}(1/t)$. We first observe that Lemma \ref{pseudo-poincare 1} yields
\[
\| P_t f -f \|_{L^p(X,\mu)}\le h(t) \| f \|_{p,h}.
\]
Thus, it remains to prove that the seminorm $\| \cdot \|_{p,h}$ satisfies condition $(\mathrm{H}_p^\rho)$ in~\cite[Corollary 10.7]{BCLS95} for some $\rho >0$. Let us denote $f_k:=(f-2^k)_+\wedge 2^k$, $k\in\mathbb{Z}$. Verbatim to the proof of~\cite[Lemma 2.6]{ARB21} (with $\rho=2$ there) one proves that
\begin{equation*}
    \sum_{k \in \mathbb{Z}} \int_X P_t(|f_k-f_k(y)|^p)(y)\,d\mu(y)\leq 2(p+1) \int_X\int_X P_t(|f-f(y)|^p)(y)\,d\mu(y)
\end{equation*}
whence
\begin{multline*}
\liminf_{t \to 0^+} \left(\sum_{k \in \mathbb{Z}}  \frac{1}{h(t)} \int_X P_t(|f_k-f_k(y)|^p)(y)\,d\mu(y) \right) \\
\le  2(p+1)\liminf_{t \to 0^+} \frac{1}{h(t)^p} \int_X\int_X P_t(|f-f(y)|^p)(y)\,d\mu(y).
\end{multline*}
By virtue of Assumption~\ref{A:PPI} and the superadditivity of the $\liminf$ it thus holds that
\begin{align*}
\sum_{k\in\mathbb{Z}}\|f_k\|_{p,h}^p&
\leq \sum_{k \in \mathbb{Z}}  \liminf_{t \to 0^+} \frac{1}{h(t)^p} \int_X P_t(|f_k-f_k(y)|^p)(y)\,d\mu(y) \\
&\le 2(p+1) \liminf_{t \to 0^+} \frac{1}{h(t)^p} \int_X P_t(|f_k-f_k(y)|^p)(y)\,d\mu(y)\leq 2(p+1)\|f\|_{p,h}^p.
\end{align*}
\end{proof}

\begin{remark}
Theorem~\ref{T:pol} generalizes~\cite[Theorem 6.9]{ABCRST1}, for which $V(s)=Cs^\beta$, $h(s)=s^\alpha$ and $\phi(s)= s^{\frac{\beta}{\beta -p\alpha }}$.
\end{remark}

\begin{remark}
  In the case $p=1$, Theorem~\ref{T:pol} also yields an isoperimetric inequality although under stronger assumptions:  Given a Young function $\phi$ satisfying~\eqref{E:Cond_Delta2} and~\eqref{E:ass_hV}, applying Lemma~\ref{P:Orlicz_properties} and Theorem~\ref{T:pol} with $f=1_E$, where $E\subset X$ is of finite $h$-perimeter, we get
    \begin{equation*}
        {\rm Per}_h(E)\geq C\mu(E) \psi^{-1}\bigg(\frac{1}{\mu(E)}\bigg),
    \end{equation*}
    where $\psi$ is the Young conjugate of $\phi$.
\end{remark}

\section{Korevaar-Schoen-type characterization}\label{S:KS_char}
From this section on, the underlying space $(X,\mu)$ is equipped with a distance $d$. We assume that $(X,d)$ is locally compact, complete, and that $\mu$ is a Radon measure. Open metric balls will be denoted
\[
B(x,r)= \{ y \in X, d(x,y)<r \}.
\] 

We will always assume that the measure $\mu$ is doubling and positive in the sense that there exists a constant $C>0$ such that for every $x \in X, r>0$,
\begin{equation}\label{A:VD}
0< \mu (B(x,2r)) \le C \mu(B(x,r)) <\infty.\tag{$\rm VD$}
\end{equation}


The aim of this section is to obtain functional inequalities in terms of the Orlicz spaces discussed earlier. Due to the nature of the underlying spaces that will be considered, we introduce the following notation to spell out the necessary assumptions.

\medskip

\textbf{Notation:} For $\alpha,\beta > 0$, define
\[
\sigma_{\alpha,\beta} (x):=x^\alpha 1_{[0,1]}(x) +x^\beta 1_{(1,+\infty)}(x).
\]

We record several basic properties of $\sigma_{\alpha,\beta}$ that will repeatedly be used throughout the next sections.

\begin{proposition}\label{P:props_sigma}
Let $\alpha,\beta,\tilde{\alpha},\tilde{\beta},q>0$.
\begin{enumerate}[wide=0em,label={\rm (\roman*)}]
\item $\sigma_{\alpha,\beta}$ is an increasing homeomorphism with $\sigma^{-1}_{\alpha,\beta} (x)=\sigma_{\frac{1}{\alpha},\frac{1}{\beta}}(x)$.
\item $(\sigma_\alpha(x))^q=\sigma_{\alpha q}(x)$.
\item  $\sigma_{\alpha,\beta}(x) \cdot \sigma_{\tilde{\alpha},\tilde{\beta}}(x)=\sigma_{\alpha+\tilde{\alpha},\beta+\tilde{\beta}}(x)$.
\item For $0< r \le R$,
\[
\left( \frac{R}{r} \right)^{ \alpha \wedge \beta} \le \frac{\sigma_{\alpha,\beta}(R)}{\sigma_{\alpha,\beta}(r)} \le \left( \frac{R}{r} \right)^{ \alpha \vee \beta}.
\]
\end{enumerate}
\end{proposition}

For the remainder of this section we add the following assumptions to those made in Section~\ref{S:hBesov_spaces}. Hereby we use the convention that $A\simeq B$ when there is a constant $C>0$ such that $C^{-1}A\leq B\leq CA$, and $A\asymp B$ if $cA\leq B\leq CA$ with possibly different and irrelevant constants $c,C>0$.

\begin{assumption}\label{A:VR_cond}
There exist $\alpha_1,\alpha_2 \ge 1$ such that 
\begin{equation}\label{E:VR_cond}
\mu(B(x,r))\simeq \sigma_{\alpha_1,\alpha_2} (r)
\end{equation}
for all $x\in X$ and $r>0$.
\end{assumption}
%

\begin{assumption}\label{A:HKE}
The heat semigroup $\{P_t\}_{t\geq 0}$ has a heat kernel satisfying the estimates
\begin{equation}\label{E:assump_HKE}
    p_t(x,y) \asymp
    \begin{cases}
    \frac{1}{\sigma_{\alpha_1/\beta_1,\alpha_2/\beta_2 }(t)}\exp\Big(\!\!-c_1\Big(\frac{d(x,y)^{\beta_1}}{t}\Big)^{\frac{1}{\beta_1-1}}\Big), \quad t < d(x,y), \\
    \frac{1}{\sigma_{\alpha_1/\beta_1,\alpha_2/\beta_2 }(t)}\exp\Big(\!\!-c_2\Big(\frac{d(x,y)^{\beta_2}}{t}\Big)^{\frac{1}{\beta_2-1}}\Big), \quad t \ge d(x,y),
    \end{cases}
\end{equation}
for $x,y\in X$, where $\alpha_i,\beta_i>0$ satisfy 
\begin{equation}\label{E:assump_exps_HKE}
2\leq \beta_i\leq 1+\alpha_i
\end{equation}
for each $i=1,2$.
\end{assumption}

Note that the estimate~\eqref{E:assump_HKE} can more concisely be written as
\begin{align}\label{concise upper bound}
 p_t(x,y) \asymp \frac{1}{\sigma_{\alpha_1/\beta_1,\alpha_2/\beta_2 }(t)} \exp\bigg(\!\!-c t \sigma_{\beta_2/(\beta_2-1),\beta_1/(\beta_1-1) } \bigg( \frac{d(x,y)}{t}\bigg)\bigg).
\end{align}

\subsection{Function space characterization}\label{SS:KS_char}

We now connect the Besov space $\mathbf{B}^{p,h}(X)$ with the Korevaar-Schoen type space arising from  Korevaar-Schoen type energies. For an increasing homeomorphism $\Psi: [0,\infty)\to [0,\infty)$ such that $\Psi(0)=0$ we define the functional
\begin{align}
E_{p,\Psi}(f,r):=\frac{1}{\Psi(r)^p}\int_X\fint_{B(y,r)}|f(x)-f(y)|^p d\mu(x)\,d\mu(y)
\end{align}
and the function space
\begin{align*}
KS^{p,\Psi}(X):=\{f\in L^p(X,\mu)\colon \sup_{r>0}E_{p,\Psi}(f,r)<\infty \}
\end{align*}
equipped with the seminorm
\begin{equation}\label{E:def_KS_snorm}
    \|f\|_{KS^{p,\Psi}}^p:=\sup_{r>0 }E_{p,\Psi}(f,r).
\end{equation}

Generalizing ideas in~\cite{KS05} and~\cite{ABCRST3} we find the correspondence between the parameters that give equivalence of the spaces under the heat kernel estimates~\eqref{E:assump_HKE}.

\begin{theorem}\label{T:KS_char}
Assume that the underying space satisfies the assumptions from Section~\ref{S:KS_char}. Let $\nu_1,\nu_2 >0$. Then, there exist constants $c, C>0$ such that
\begin{equation}\label{E:KS_char}
        c\sup_{r>0}E_{p,\sigma_{\nu_1 \beta_1,\nu_2 \beta_2 }}(f,r)\leq \|f\|_{p,\sigma_{\nu_1,\nu_2}}^p\leq C\sup_{r>0}E_{p,\sigma_{\nu_1 \beta_1,\nu_2 \beta_2 }}(f,r)
        \end{equation}
\end{theorem}
\begin{proof}
To prove the first inequality in~\eqref{E:KS_char}, Assumption~\ref{A:HKE} yields for $r,t>0$
\begin{align*}
    &\frac{1}{\sigma_{\nu_1,\nu_2}(t)^p}\int_XP_t(|f-f(y)|^p)(y)\,d\mu(y)\\
    =&\frac{1}{\sigma_{\nu_1,\nu_2}(t)^p}\int_X\int_X p_t(x,y)|f(x)-f(y)|^pd\mu(y)\,d\mu(x)\\
    \ge &\frac{1}{\sigma_{\nu_1,\nu_2}(t)^p}\int_X\int_{B(x,r)} p_t(x,y)|f(x)-f(y)|^pd\mu(y)\,d\mu(x) \\
    \ge & \frac{c_1}{\sigma_{\nu_1,\nu_2}(t)^p\sigma_{\alpha_1/\beta_1,\alpha_2/\beta_2 }(t)}\exp\Big(\!\!-c_2 t \sigma_{\beta_2/(\beta_2-1),\beta_1/(\beta_1-1) } \left( \frac{r}{t}\right)\Big)\int_X\int_{B(x,r)} |f(x)-f(y)|^pd\mu(y)\,d\mu(x)
\end{align*}

Choosing $r=\sigma_{1/\beta_1,1/\beta_2} (t)$ we obtain
\begin{align*}
 &    \frac{1}{\sigma_{\nu_1,\nu_2}(t)^p}\int_XP_t(|f-f(y)|^p)(y)\,d\mu(y)\\
    \ge & \frac{c}{\sigma_{\nu_1,\nu_2}(t)^p\sigma_{\alpha_1/\beta_1,\alpha_2/\beta_2 }(t)}\int_X\int_{B(x,\sigma_{1/\beta_1,1/\beta_2} (t))} |f(x)-f(y)|^pd\mu(y)\,d\mu(x)
 \end{align*}  
 Therefore we have 
\begin{align*}
   \|f\|_{p,\sigma_{\nu_1,\nu_2}}^p  &\ge  \sup_{t >0} \frac{c}{\sigma_{\nu_1,\nu_2}(t)^p\sigma_{\alpha_1/\beta_1,\alpha_2/\beta_2 }(t)}\int_X\int_{B(x,\sigma_{1/\beta_1,1/\beta_2} (t))} |f(x)-f(y)|^pd\mu(y)\,d\mu(x) \\
     & =c\sup_{r>0}E_{p,\sigma_{\nu_1 \beta_1,\nu_2 \beta_2 }}(f).
 \end{align*}  
 
We now turn to the upper bound. Fixing $r>0$, we set
\begin{align}
A(t,r)&:=\int_X\int_{X\setminus B(y,r)}p_{t}(x,y)|f(x)-f(y)|^{p}\,d\mu(x)\,d\mu(y),\label{E:A(t)}\\
B(t,r)&:=\int_X \int_{B(y,r)}p_{t}(x,y)|f(x)-f(y)|^{p}\,d\mu(x)\,d\mu(y),\label{E:B(t)}
\end{align}
so that $ \int_X \int_X |f(x)-f(y) |^p p_t (x,y) d\mu(x) d\mu(y) =A(t,r)+B(t,r)$.

\medskip

Applying~\eqref{concise upper bound} and the inequality $|f(x)-f(y)|^{p}\leq 2^{p-1}(|f(x)|^{p}+|f(y)|^{p})$ yields

\begin{align*}
&A(t,r)\leq\frac{c}{\sigma_{\alpha_1/\beta_1,\alpha_2/\beta_2 }(t)}\int_X\int_{X\setminus B(y,r)} \exp\Big(\!\!-C t \sigma_{\beta_2/(\beta_2-1),\beta_1/(\beta_1-1) } \left( \frac{d(x,y)}{t}\right)\Big)|f(y)|^{p}\,d\mu(x)\,d\mu(y)
\notag\\
&=
\frac{c}{\sigma_{\alpha_1/\beta_1,\alpha_2/\beta_2 }(t)}\sum_{k=1}^{\infty}\int_X\int_{B(y,2^{k}r)\setminus B(y,2^{k-1}r)}\exp\Big(\!\!-C t \sigma_{\beta_2/(\beta_2-1),\beta_1/(\beta_1-1) } \left( \frac{d(x,y)}{t}\right)\Big)|f(y)|^{p}\,d\mu(x)\,d\mu(y)
\notag\\
&\leq
\frac{c}{\sigma_{\alpha_1/\beta_1,\alpha_2/\beta_2 }(t)}\sum_{k=1}^{+\infty}\int_X\mu\bigl(B(y,2^{k}r)\bigr)\exp\Big(\!\!-C t \sigma_{\beta_2/(\beta_2-1),\beta_1/(\beta_1-1) } \left( \frac{2^{k-1}r}{t}\right)\Big)|f(y)|^{p}\,d\mu(y)
\notag\\
&\leq
\frac{c}{\sigma_{\alpha_1/\beta_1,\alpha_2/\beta_2 }(t)}\sum_{k=1}^{+\infty}\sigma_{\alpha_1,\alpha_2}(2^k r) \|f\|_{L^{p}}^{p}\exp\Big(\!\!-C t \sigma_{\beta_2/(\beta_2-1),\beta_1/(\beta_1-1) } \left( \frac{2^{k-1}r}{t}\right)\Big)
\notag\\
&\leq
\frac{c}{\sigma_{\alpha_1/\beta_1,\alpha_2/\beta_2 }(t)}\sum_{k=1}^{+\infty}\sigma_{\alpha_1,\alpha_2}(2^k r) \|f\|_{L^{p}}^{p}\exp\Big(\!\!-C t \sigma_{\beta_2/(\beta_2-1),\beta_1/(\beta_1-1) } \left( \frac{2^{k-1}r}{t}\right)\Big)
\notag\\
&\le 
\frac{c\|f\|_{L^{p}}^{p}}{\sigma_{\alpha_1/\beta_1,\alpha_2/\beta_2 }(t)} \int_{r}^{+\infty}\sigma_{\alpha_1,\alpha_2}(u)\exp\Big(\!\!-C t \sigma_{\beta_2/(\beta_2-1),\beta_1/(\beta_1-1) } \left( \frac{u}{t}\right)\Big) \frac{du}{u} \\
&\le \frac{c\|f\|_{L^{p}}^{p}}{\sigma_{\alpha_1/\beta_1,\alpha_2/\beta_2 }(t)} \int_{0}^{+\infty}\sigma_{\alpha_1,\alpha_2}(u)\exp\Big(\!\!-\frac{C}{2} t \sigma_{\beta_2/(\beta_2-1),\beta_1/(\beta_1-1) } \left( \frac{u}{t}\right)\Big) \frac{du}{u} \\
 &\, \, \, \, \, \, \,  \exp\Big(\!\!-\frac{C}{2} t \sigma_{\beta_2/(\beta_2-1),\beta_1/(\beta_1-1) } \left( \frac{r}{t}\right)\Big) \\
 & \le c\|f\|_{L^{p}}^{p}  \exp\Big(\!\!-C t \sigma_{\beta_2/(\beta_2-1),\beta_1/(\beta_1-1) } \left( \frac{r}{t}\right)\Big).
\end{align*}

On the other hand, the heat kernel estimate~\eqref{concise upper bound} allows to bound $B(t,r)$ by
\begin{align}
& \frac{c}{\sigma_{\alpha_1/\beta_1,\alpha_2/\beta_2 }(t)}\int_X \int_{B(y,r)} \exp\Big(\!\!-c t \sigma_{\beta_2/(\beta_2-1),\beta_1/(\beta_1-1) } \left( \frac{d(x,y)}{t}\right)\Big) |f(x)-f(y)|^{p}\,d\mu(x)\,d\mu(y)\notag\\
\leq &
\frac{c}{\sigma_{\alpha_1/\beta_1,\alpha_2/\beta_2 }(t)}\sum_{k=1}^{+\infty}\int\limits_X\int\limits_{B(y,2^{1-k}r)\setminus B(y,2^{-k}r)}\exp\Big(\!\!-c t \sigma_{\beta_2/(\beta_2-1),\beta_1/(\beta_1-1) } \left( \frac{d(x,y)}{t}\right)\Big) |f(x)-f(y)|^{p}\,d\mu(x)\,d\mu(y)\notag\\
\leq &
\frac{c}{\sigma_{\alpha_1/\beta_1,\alpha_2/\beta_2 }(t)}\sum_{k=1}^{+\infty}\int\limits_X\int\limits_{B(y,2^{1-k}r)\setminus B(y,2^{-k}r)}\exp\Big(\!\!-c t \sigma_{\beta_2/(\beta_2-1),\beta_1/(\beta_1-1) } \left( \frac{2^{-k}r}{t}\right)\Big) |f(x)-f(y)|^{p}\,d\mu(x)\,d\mu(y)\notag\\
\leq &
\frac{c}{\sigma_{\alpha_1/\beta_1,\alpha_2/\beta_2 }(t)}\sum_{k=1}^{+\infty} \exp\Big(\!\!-c t \sigma_{\beta_2/(\beta_2-1),\beta_1/(\beta_1-1) } \left( \frac{2^{-k}r}{t}\right)\Big) \int\limits_X\int\limits_{B(y,2^{1-k}r)} |f(x)-f(y)|^{p}\,d\mu(x)\,d\mu(y)\notag\\
\leq & c
\sum_{k=1}^{+\infty} \frac{\exp\Big(\!\!-c t \sigma_{\beta_2/(\beta_2-1),\beta_1/(\beta_1-1) } \left( \frac{2^{-k}r}{t}\right)\Big) \sigma_{\nu_1\beta_1,\nu_2\beta_2}(2^{1-k}r)^p\sigma_{\alpha_1,\alpha_2 }(2^{1-k}r)}{\sigma_{\alpha_1/\beta_1,\alpha_2/\beta_2 }(t)} E_{p,\sigma_{\nu_1 \beta_1,\nu_2 \beta_2 }}(f,2^{1-k}r)\notag\\
\leq & c \sup_{s \in (0,r]} E_{p,\sigma_{\nu_1 \beta_1,\nu_2 \beta_2 }}(f,s)
\sum_{k=1}^{+\infty} \frac{\exp\Big(\!\!-c t \sigma_{\beta_2/(\beta_2-1),\beta_1/(\beta_1-1) } \left( \frac{2^{-k}r}{t}\right)\Big) \sigma_{\nu_1\beta_1,\nu_2\beta_2}(2^{1-k}r)^p\sigma_{\alpha_1,\alpha_2 }(2^{1-k}r)}{\sigma_{\alpha_1/\beta_1,\alpha_2/\beta_2 }(t)} \notag\\
\leq &c  \sigma_{\nu_1,\nu_2}(t)^p \sup_{s \in (0,r]} E_{p,\sigma_{\nu_1 \beta_1,\nu_2 \beta_2 }}(f,s).
\end{align}
Combining both estimates it follows that
\begin{align*}
  & \int_X \int_X |f(x)-f(y) |^p p_t (x,y) d\mu(x) d\mu(y) \\
 \le & c_1  \exp\Big(\!\!-C_1 t \sigma_{\beta_2/(\beta_2-1),\beta_1/(\beta_1-1) } \left( \frac{r}{t}\right)\Big) \|f\|_{L^{p}}^{p}+c_2  \sigma_{\nu_1,\nu_2}(t)^p \sup_{s \in (0,r]} E_{p,\sigma_{\nu_1 \beta_1,\nu_2 \beta_2 }}(f,s),
 \end{align*}
which yields
\begin{multline*}
    \sup_{t>0}\frac{1}{\sigma_{\nu_1,\nu_2}(t)^p} \int_X \int_X |f(x)-f(y) |^p p_t (x,y) d\mu(x) d\mu(y) \\
\le c \sup_{s \in (0,r]} E_{p,\sigma_{\nu_1 \beta_1,\nu_2 \beta_2 }}(f,s)+\frac{C}{\sigma_{\nu_1 \beta_1,\nu_2 \beta_2 }(r)^p} \|f\|_{L^{p}(X,\mu)}^{p}.
\end{multline*}
The proof is completed by letting $r \to +\infty$.
\end{proof}

From Theorem~\ref{T:KS_char} one now directly finds the connection between the Korevaar-Schoen-type spaces in this section and the Besov spaces from the previous section.

\begin{corollary}\label{C:KS_char}
Let the underlying space satisfy the assumptions from Section~\ref{S:KS_char}. Then, for any $\nu_1,\nu_2>0$ it holds that 
\begin{equation}\label{E:KS_char_a}
KS^{p,\sigma_{\nu_1,\nu_2}}(X)=\mathbf{B}^{p,\sigma_{\nu_1/\beta_1,\nu_2/\beta_2}}(X),
\end{equation}
equivalently
\begin{equation}\label{E:KS_char_b}
KS^{p,\sigma_{\nu_1\beta_1,\nu_2\beta_2}}(X)=\mathbf{B}^{p,\sigma_{\nu_1,\nu_2}}(X),
\end{equation}
with equivalent seminorms.
\end{corollary}

\subsection{Weak Bakry-\'Emery condition}

Assuming the volume regular condition~\eqref{E:VR_cond} it has recently been proved in~\cite{GY24} that, under suitable conditions, the heat semigroup $\{P_t\}_{t\geq 0}$ satisfies a H\"older estimate that is generalized in the next definition.

\begin{definition}
The semigroup $\{P_t\}_{t\geq 0}$ is said to satisfy the weak Bakry-\'Emery condition ${\rm wBE}(\kappa_1,\kappa_2)$ if there exist $C_{\rm wBE},\kappa_1,\kappa_2>0$ such that
\begin{equation}\label{E:wBE}
    |P_tf(x)-P_tf(y)|\leq C_{\rm wBE}\frac{\sigma_{\kappa_1,\kappa_2}\big(d(x,y)\big)}{\sigma_{\frac{\kappa_1}{\beta_1},\frac{\kappa_2}{\beta_2}}(t)}\|f\|_{L^\infty}
\end{equation}
for any $x,y\in X$, $t>0$ and $f\in L^\infty (X,\mu)$. 
\end{definition}

This property will play a pivotal role in proving the continuity of the heat semigroup in  a given Besov class, see Proposition \ref{T:continuity}. We prove below that it is in fact satisfied, together with the previous assumptions, in a large class of spaces.

\begin{prop}\label{P:wBE_product}
    Let $(X,d,\mu)$ be a metric Dirichlet space as before which has a heat kernel that satisfies the estimates \eqref{E:assump_HKE} with $\alpha_i <\beta_i$, $i=1,2$. Then, for every $n \ge 1$, the semigroup on the product metric Dirichlet space $(X^n,d_{X^n},\mu^{\otimes n})$ satisfies ${\rm wBE}(\beta_1-\alpha_1,\beta_2-\alpha_2)$. Here, the distance on the product space is defined by
    \[
    d_{X^n}(\mathbf{x},\mathbf{y}):=\sum_{i=1}^nd(x_i,y_i).
    \]
\end{prop}

\begin{proof}
The case $n=1$ follows from~\cite{GY24}. The general case proceeds by induction, and we show the case $n=2$, the general case following the same line of arguments. Let $f\in L^{\infty}(X^2, \mu^{\otimes 2})$.   Given $\mathbf x=(x_1,y_1), \mathbf x'=(x_2,y_2)\in X^2$ we can estimate the change in $P_t^{X^2}f$ due to changing a single component of $\mathbf x$ using $wBE(\beta_1-\alpha_1,\beta_2-\alpha_2)$ on $X$:
\begin{align*}
&\left|P_t^{X^2} f(x_1, y_1)-P_t^{X^2} f(x_1, y_2) \right|
\\ 
= & 
\left|\int_X p_t(x_1,z)P_t[f (z,\cdot)](y_1) d\mu(z)-\int_X p_t(x_1,z)P_t[f (z,\cdot)](y_2) d\mu(z)\right|
\\ 
\le & 
\int_X p_t(x_1,z)\left| P_t[f (z,\cdot)](y_1) - P_t[f (z,\cdot)](y_2)\right| d\mu(z)
\\ 
\le &
C \frac{\sigma_{\beta_1-\alpha_1,\beta_2-\alpha_2}\big(d(y_1,y_2)\big)}{\sigma_{\frac{\beta_1-\alpha_1}{\beta_1},\frac{\beta_2-\alpha_2}{\beta_2}}(t)} \|f\|_{L^{\infty}(X^2)}.
\end{align*}
Estimating the change in the other component in the same manner and summing over components it follows that
\begin{align*}
\bigl|P_t^{X^n} f(\mathbf x)-P_t^{X^n} f(\mathbf x')\bigr|
& \le \frac{C}{\sigma_{\frac{\beta_1-\alpha_1}{\beta_1},\frac{\beta_2-\alpha_2}{\beta_2}}(t)}  \bigg(\sum_{i=1}^2\sigma_{\beta_1-\alpha_1,\beta_2-\alpha_2}\big(d(x_i,y_i)\big)\bigg) \|f\|_{L^{\infty}(X^n)} \\
 & \le
C \frac{\sigma_{\beta_1-\alpha_1,\beta_2-\alpha_2}(d_{X^2}(\mathbf x,\mathbf x'))}{\sigma_{\frac{\beta_1-\alpha_1}{\beta_1},\frac{\beta_2-\alpha_2}{\beta_2}}(t)} \|f\|_{L^{\infty}(X^n)},
\end{align*}
where we used in the last inequality the fact that there exists a constant $C>0$ such that for every $a_1,a_2 \ge 0$
\[
\sum_{i=1}^2 \sigma_{\beta_1-\alpha_1,\beta_2-\alpha_2} (a_i) \le C \sigma_{\beta_1-\alpha_1,\beta_2-\alpha_2} \left(\sum_{i=1}^2 a_i\right).
\]
\end{proof}

We point out that the product metric Dirichlet space $(X^n,d_{X^n},\mu^{\otimes n})$ satisfies the assumptions from this Section~\ref{S:KS_char}. For instance, the heat kernel is
\[
p_t^{X^n}(\mathbf{x},\mathbf{y})=p_t(x_1,y_1)\cdots p_t(x_n,y_n).
\]
Thus, it follows that
\begin{align*}
p_t^{X^n}(\mathbf{x},\mathbf{y}) &\asymp \frac{1}{\sigma_{\alpha_1/\beta_1,\alpha_2/\beta_2 }(t)^n} \exp\bigg(\!\!-c \sum_{i=1}^nt \sigma_{\beta_2/(\beta_2-1),\beta_1/(\beta_1-1) } \bigg( \frac{d(x_i,y_i)}{t}\bigg)\bigg) \\
 &\asymp \frac{1}{\sigma_{n\alpha_1/\beta_1,n\alpha_2/\beta_2 }(t)} \exp\bigg(\!\!-c t \sigma_{\beta_2/(\beta_2-1),\beta_1/(\beta_1-1) } \bigg( \frac{\sum_{i=1}^n d(x_i,y_i)}{t}\bigg)\bigg),
\end{align*}
where in the last inequality we use the fact that there exist constants $c,C>0$ such that for every $a_1,\cdots,a_n \ge 0$
\[
c \sigma_{\beta_1-\alpha_1,\beta_2-\alpha_2} \left(\sum_{i=1}^n a_i\right) \le \sum_{i=1}^n \sigma_{\beta_1-\alpha_1,\beta_2-\alpha_2} (a_i) \le C \sigma_{\beta_1-\alpha_1,\beta_2-\alpha_2} \left(\sum_{i=1}^n a_i\right).
\]

\begin{remark}\label{R:examples}
    The weak Bakry-\'Emery condition~\eqref{E:wBE} has been proved to be satisfied for Sierpinski gasket cable system and several fractal blow-ups in~\cite[Section 9]{GY24}. For instance, in the case of the Sierpinski cable system, 
    \begin{equation*}
    \alpha_1=1<2=\beta_1,\quad\alpha_2=\alpha:=\frac{\log 8}{\log 3}<\beta=\beta_2\quad\text{and}\quad\kappa_i=\beta_i-\alpha_i,\; i=1,2,
    \end{equation*}
    where $\beta>\frac{\log 8}{\log 3}$ is a suitable parameter depending on the walk dimension of the standard Sierpinski carpet. 
    In view of Proposition~\ref{P:wBE_product}, ${\rm wBE}(\kappa_1,\kappa_2)$ is also satisfied by non-trivial products of these fractals.
\end{remark}

\begin{remark}\label{R:open_Q}
    The validity of the weak Bakry-\'Emery condition~\eqref{E:wBE} on fractal-like manifolds as considered  e.g. in~\cite{BBK06} is still an open problem. 
    In that direction, we point out that, thanks to~\cite[Corollary 2.3]{CCFR17}, a weak Bakry-\'Emery condition would follow from the uniform gradient estimate on the heat kernel $| \nabla_x p_t (x,y)| \le C | \nabla_y p_t (x,y)| $.
\end{remark}

\subsection{Continuity of the semigroup in the Besov class}
We now turn to study some analytic properties of the heat semigroup $\{P_t\}_{t\geq 0}$ as an operator in a suitable Besov space from Section~\ref{S:hBesov_spaces}.
\begin{theorem}\label{T:continuity}
    Let $2\leq p<\infty$. Under the assumptions of Section~\ref{S:KS_char}, there exists $C>0$ such that for every $f \in L^p(X,\mu)$
    \begin{equation}\label{E:continuity}
    \|P_tf\|_{p,\Psi_p}\leq  
    \frac{C}{\Psi_p(t)}\|f\|_{L^p(X,\mu)}
    \end{equation}
    for any $t>0$, where 
    \begin{equation*}
    \begin{aligned}
    \Psi_p(r)&=\sigma_{(1-\frac{2}{p})\frac{\kappa_1}{\beta_1}+\frac{1}{p},(1-\frac{2}{p})\frac{\kappa_2}{\beta_2}+\frac{1}{p}}(r).
    \end{aligned}
    \end{equation*}
\end{theorem}

\begin{remark}\label{R:sanity_check_continuity}
The estimate~\eqref{E:continuity} is a generalization of~\cite[Theorem 3.9]{ABCRST3}. Indeed, in the case of standard sub-Gaussian heat kernel estimates and the standard volume doubling property, $\alpha_1=\alpha_2=d_H$ and $\beta_1=\beta_2=d_W$. Moreover, the weak Bakry-\'Emery condition ${\rm wBE}(\kappa)$ from~\cite[Definition 3.1]{ABCRST3} corresponds to $\kappa_1=\kappa_2=\kappa$, whence 
\[
\Psi_p(r)=r^{(1-\frac{2}{p})\frac{\kappa}{d_W}+\frac{1}{p}}.
\]
\end{remark}
\begin{proof}[Proof of Theorem~\ref{T:continuity}]
    We follow ideas from an updated proof of~\cite[Theorem 3.9]{ABCRST3}. Let $r>0$ and define the measure $d\mu_r(x,y):=\mathbf{1}_{\{d(x,y)<r\}}d\mu(x)\otimes d\mu(y)$ as well as the linear operator 
    \begin{equation}\label{E:aux_op_Pt}
    \begin{aligned}
    \mathcal{P}_t\colon& L^2(X,\mu)\longrightarrow L^2(X\times X,d\nu_r)\\    
    &\qquad f\quad\longmapsto\quad P_tf(x)-P_tf(y).
    \end{aligned}
    \end{equation}
    From the lower bound in Assumption~\ref{A:HKE} and the volume regular condition~\eqref{E:VR_cond} it follows that 
    \begin{multline*}
        \int_X\int_X(P_tf(x)-P_tf(y))^2d\nu_r(x,y)=\int_X\int_{B(x,r)}(P_tf(x)-P_tf(y))^2d\mu(y)\,d\mu(x)\\
        \leq Cr^{\alpha_1}\mathbf{1}_{\{0<r<1\}}(r)\int_X\int_{B(x,r)}p_{r^{\beta_1}}(x,y)(P_tf(x)-P_tf(y))^2d\mu(y)\,d\mu(x)\\
        + Cr^{\alpha_2}\mathbf{1}_{\{r>1\}}(r)\int_X\int_{B(x,r)}p_{r^{\beta_2}}(x,y)(P_tf(x)-P_tf(y))^2d\mu(y)\,d\mu(x)
    \end{multline*}
    Further, by virtue of~\cite[Theorem 5.1]{ABCRST1} we have the estimate
    \[
    \sup_{s>0}s^{-1/2}\bigg(\int_X\int_Xp_s(x,y)(P_tf(x)-P_tf(y))^2d\mu(y)\,d\mu(x)\bigg)^{1/2}\leq Ct^{-1/2}\|f\|_{L^2(X,\mu)},
    \]
    which applied to the previous bound (with $s=r^{\beta_1}$ and $s=r^{\beta_2}$ respectively) yields
    \begin{equation*}
        \int_X\int_X(P_tf(x)-P_tf(y))^2d\nu_r(x,y)
       \leq C\sigma_{\alpha_1+\beta_1,\alpha_2+\beta_2}(r)t^{-1}\|f\|_{L^2(X,\mu)}^2
    \end{equation*}
    and hence
    \begin{equation*}
    \|\mathcal{P}_t\|_{L^2(X,\mu)\to L^2(X\times X,d\nu_r)}<\frac{C}{t^{1/2}}\sigma_{\alpha_1+\beta_1,\alpha_2+\beta_2}(r)^{1/2}.    
    \end{equation*}
    In addition, the weak Bakry-\'Emery estimate~\eqref{E:wBE} implies 
    \begin{equation*}
        \|\mathcal{P}_t\|_{L^\infty(X,\mu)\to L^\infty(X\times X,d\nu_r)}\\
        <C_{\rm wBE}\sigma_{\kappa_1,\kappa_2}(r)\sigma_{\frac{-\kappa_1}{\beta_1},\frac{-\kappa_2}{\beta_2}}(t).
    \end{equation*}    
    By virtue of the Riesz-Thorin interpolation theorem and Proposition~\ref{P:props_sigma}, it follows that
    \begin{align}
        &\|\mathcal{P}_t\|_{L^p(X,\mu)\to L^p(X\times X,d\nu_r)}\notag\\
        &<C\big(\sigma_{\alpha_1+\beta_1,\alpha_2+\beta_2}(r)t^{-1}\big)^{\frac{1}{p}}
       \big(\sigma_{\kappa_1,\kappa_2}(r)\sigma_{\frac{-\kappa_1}{\beta_1},\frac{-\kappa_2}{\beta_2}}(t) \big)^{1-\frac{2}{p}}\notag\\
       &=C\sigma_{\frac{\alpha_1+\beta_1}{p},\frac{\alpha_2+\beta_2}{p}}(r)\sigma_{\kappa_1(1-\frac{2}{p}),\kappa_2(1-\frac{2}{p})}(r)t^{-\frac{1}{p}}\sigma_{\frac{\kappa_1}{\beta_1}(\frac{2}{p}-1),\frac{\kappa_2}{\beta_2}(\frac{2}{p}-1)}(t)\notag\\
       &=C\sigma_{\frac{1}{p}(\alpha_1+\beta_1+(p-2)\kappa_1),\frac{1}{p}(\alpha_2+\beta_2+(p-2)\kappa_2)}(r)\sigma_{\frac{\kappa_1}{\beta_1}(\frac{2}{p}-1)-\frac{1}{p},\frac{\kappa_2}{\beta_2}(\frac{2}{p}-1)-\frac{1}{p}}(t)\label{E:aux_po_Lp_est}
    \end{align}
    for any $2\leq p<\infty$. In other words, and using the volume regularity assumption~\eqref{E:VR_cond},
    \begin{multline*}
       \frac{\sigma_{\alpha_1,\alpha_2}(r)^{1/p}}{\sigma_{\frac{1}{p}(\alpha_1+\beta_1+(p-2)\kappa_1),\frac{1}{p}(\alpha_2+\beta_2+(p-2)\kappa_2)}(r)}\bigg(\int_X\fint_{B(x,r)}|P_tf(x)-P_tf(y)|^pd\mu(y)\,d\mu(x)\bigg)^{1/p}\\
       \leq C \sigma_{\frac{\kappa_1}{\beta_1}(\frac{2}{p}-1)-\frac{1}{p},\frac{\kappa_2}{\beta_2}(\frac{2}{p}-1)-\frac{1}{p}}(t)\|f\|_{L^p(X,\mu)}.
    \end{multline*}
    Taking $\sup_{r>0}$ on both sides of the above inequality and applying Theorem~\ref{T:KS_char} with $\nu_i=\frac{1}{p}(\beta_i+(p-2)\kappa_i)$ it finally follows that
    \begin{equation*}
        \|f\|_{p,\sigma_{(1-\frac{2}{p})\frac{\kappa_1}{\beta_1}+\frac{1}{p},(1-\frac{2}{p})\frac{\kappa_2}{\beta_2}+\frac{1}{p}}}
        \leq C \sigma_{\frac{\kappa_1}{\beta_1}(\frac{2}{p}-1)-\frac{1}{p},\frac{\kappa_2}{\beta_2}(\frac{2}{p}-1)-\frac{1}{p})}(t)\|f\|_{L^p(X,\mu)}
    \end{equation*}
    as we wanted to prove.
 \end{proof}

 \subsection{Pseudo Poincar\'e inequality}
 An important tool that will later on lead to proving the weak monotonicity from Assumption~\ref{A:PPI} is the following Pseudo Poincar\'e inequality.
\begin{proposition}\label{P:pseudoPI}
Under the main assumptions of this section, there exists $C>0$ such that for every $f \in \mathbf{B}^{p,\Psi_p}(X)$ and $t \ge 0$ 
    \begin{equation}\label{E:pseudoPI}
        \|P_tf-f\|_{L^p(X,\mu)}\leq C\Psi_p(t)\bigg(\liminf_{\tau\to 0^+}\frac{1}{\tau^{(1-\frac{2}{p})\frac{\kappa_1}{\beta_1}+\frac{1}{p}}}\int_XP_\tau(|f-f(y)|)^pd\mu(y)\bigg)^{1/p}
    \end{equation}
    with 
    \[
    \Psi_p(t)=\sigma_{(1-\frac{2}{p})\frac{\kappa_1}{\beta_1}+\frac{1}{p},(1-\frac{2}{p})\frac{\kappa_2}{\beta_2}+\frac{1}{p}}(t)
    \]
    for $1\leq p\leq 2$.
\end{proposition}

\begin{remark}\label{R:sanity_check_pseudoPI}
    The estimate~\eqref{E:pseudoPI} is a generalization of~\cite[Proposition 3.10]{ABCRST3}. Again, in the case of standard sub-Gaussian heat kernel estimates and the standard volume doubling property, $\alpha_1=\alpha_2=d_H$ and $\beta_1=\beta_2=d_W$. Under the weak Bakry-\'Emery condition ${\rm wBE}(\kappa)$, it holds that $\frac{\kappa_1}{\beta_1}=\frac{\kappa_2}{\beta_2}=\frac{\kappa}{d_W}$, whence
    \[
    \Psi_p(t)=t^{(1-\frac{2}{p})\frac{\kappa}{d_W}+\frac{1}{p}}
    \]
    for $1\leq p\leq 2$, which coincides with~\cite[Proposition 3.10]{ABCRST3}.
\end{remark}

\begin{proof}[Proof of Proposition~\ref{P:pseudoPI}]
    We prove the estimate by duality following a similar argument as in~\cite[Proposition 3.10]{ABCRST3}. The strong continuity of $\{P_t\}_{t\geq 0}$ in $L^1(X,\mu)$ implies
    \begin{equation}\label{E:pseudoPI_01}
    \int_X(P_tf-f)g\,d\mu=-\lim_{\tau\to 0^+}\int_0^t\frac{1}{\tau}\int_XP_sf(I-P_\tau)g\,d\mu\,ds
    \end{equation}
    for any bounded Borel measurable $g$. Further, the symmetry of the heat kernel yields
    \begin{align}
       \frac{1}{\tau}\int_XP_sf(I-P_\tau)g\,d\mu
        &=\frac{1}{\tau}\int_Xf(I-P_\tau)P_sg\,d\mu\label{E:pseudoPI_02}\\
        &=\frac{1}{2\tau}\int_X\int_Xp_\tau(x,y)(P_sg(x)-P_sg(y))(f(x)-f(y))\,d\mu(y)\,d\mu(x).\notag
    \end{align}
    In the case $1<p\leq 2$, applying H\"older's inequality and Theorem~\ref{T:continuity}, the above expression is bounded by
    \begin{align}
        &\frac{1}{\tau}\bigg(\int_X\int_Xp_\tau(x,y)|P_sg(x)-P_sg(y)|^qd\mu(y)\,d\mu(x)\bigg)^{\frac{1}{q}}\bigg(\int_X\int_Xp_\tau (x,y)|f(x)-f(y)|^pd\mu(y)\,d\mu(x)\bigg)^{\frac{1}{p}}\\
        &\leq C\frac{\Psi_q(\tau)}{\tau}\|P_sg\|_{q,\Psi_q}\bigg(\int_X\int_Xp_\tau (x,y)|f(x)-f(y)|^pd\mu(y)\,d\mu(x)\bigg)^{\frac{1}{p}}\\
        &\leq C\frac{\Psi_q(\tau)}{\tau \Psi_q(s)}\|g\|_{L^q(X,\mu)}\bigg(\int_X\int_Xp_\tau (x,y)|f(x)-f(y)|^pd\mu(y)\,d\mu(x)\bigg)^{\frac{1}{p}},\label{E:pseudoPI_05}
    \end{align}
    where 
    \[
    \Psi_q(\tau)=\sigma_{(1-\frac{2}{q})\frac{\kappa_1}{\beta_1}+\frac{1}{q},(1-\frac{2}{q})\frac{\kappa_2}{\beta_2}+\frac{1}{q}}(\tau).
    \]
    Note also, that for $\frac{1}{p}+\frac{1}{q}=1$,
    \[
    \frac{\Psi_q(\tau)}{\tau}=\sigma_{(1-\frac{2}{q})\frac{\kappa_1}{\beta_1}+\frac{1}{q}-1,(1-\frac{2}{q})\frac{\kappa_2}{\beta_2}+\frac{1}{q}-1}(\tau)
    =\sigma_{(\frac{2}{p}-1)\frac{\kappa_1}{\beta_1}-\frac{1}{p},(\frac{2}{p}-1)\frac{\kappa_2}{\beta_2}-\frac{1}{p}}(\tau)=\frac{1}{\Psi_p(\tau)}.
    \]
    Thus, integrating~\eqref{E:pseudoPI_05} over $s\in (0,t)$, taking $\liminf_{\tau\to 0^+}$ and applying Lemma~\ref{L:integral_sigma} with $\nu_1= (1-\frac{2}{q})\frac{\kappa_1}{\beta_1}+\frac{1}{q}$ and $\nu_2=(1-\frac{2}{q})\frac{\kappa_2}{\beta_2}+\frac{1}{q}$ it follows from~\eqref{E:pseudoPI_01} that
    \begin{align*}
        &\bigg|\int_X(P_tf-f)g\,d\mu\bigg|\\
        &\leq C\|g\|_{L^q(X,\mu)}\int_0^t\frac{1}{\Psi_q(s)}ds\liminf_{\tau\to 0^+} \frac{\Psi_q(\tau)}{\tau}\bigg(\int_X\int_Xp_\tau (x,y)|f(x)-f(y)|^pd\mu(y)\,d\mu(x)\bigg)^{\frac{1}{p}}\\
        &\leq C\|g\|_{L^q(X,\mu)}\Psi_p(t)\liminf_{\tau\to 0^+} \frac{1}{\Psi_p(\tau)}\bigg(\int_X\int_Xp_\tau (x,y)|f(x)-f(y)|^pd\mu(y)\,d\mu(x)\bigg)^{\frac{1}{p}},
    \end{align*}
    with
    \[
    \Psi_p(t)=\sigma_{(1-\frac{2}{p})\frac{\kappa_1}{\beta_1}+\frac{1}{p},(1-\frac{2}{p})\frac{\kappa_2}{\beta_2}+\frac{1}{p}}(t).
    \]
    Finally, by duality we arrive at~\eqref{E:pseudoPI}. 

    \medskip
    
    In the case $p=1$, we apply the weak Bakry-\'Emery estimate~\eqref{E:wBE} to bound~\eqref{E:pseudoPI_02} by
    \begin{equation}\label{E:pseudoPI_wBE}
    C_{\rm wBE}\sigma_{\frac{\kappa_1}{\beta_1},\frac{\kappa_2}{\beta_2}}(s)\frac{\|g\|_{L^\infty}}{\tau s}\int_X\int_X p_\tau(x,y)\sigma_{\kappa_1,\kappa_2}(d(x,y))|f(x)-f(y)|\,d\mu(y)\,d\mu(x).
    \end{equation}

    Now, it follows from~\cite[Lemma 2.3]{ABCRST3} that 
    \begin{equation*}
        \sigma_{\nu_1,\nu_2}(d(x,y)) p_\tau(x,y)\leq 
        \begin{cases}
            C \sigma_{\frac{\kappa_1}{\beta_1},\frac{\kappa_2}{\beta_1}}(\tau) p_{c\tau}(x,y)&\text{if }\tau<d(x,y),\\
            C \sigma_{\frac{\kappa_1}{\beta_2},\frac{\kappa_2}{\beta_2}}(\tau) p_{c\tau}(x,y)&\text{if }\tau\geq d(x,y)
        \end{cases}
    \end{equation*}
    for any $\kappa_1,\kappa_2>0$. Therefore,
    \begin{multline}\label{E:pseudoPI_03}
     \int_X\int_X p_\tau(x,y)\sigma_{\kappa_1,\kappa_2}(d(x,y))|f(x)-f(y)|\,d\mu(y)\,d\mu(x) \\
    \le C\sigma_{\frac{\kappa_1}{\beta_2},\frac{\kappa_2}{\beta_2}}(\tau) \int_X\int_{B(x,\tau)} p_{c\tau}(x,y)|f(x)-f(y)|\,d\mu(y)\,d\mu(x) \\
     + C\sigma_{\frac{\kappa_1}{\beta_1},\frac{\kappa_2}{\beta_1}}(\tau)\int_X\int_{X{\setminus}B(x,\tau)} p_{c\tau}(x,y) |f(x)-f(y)|\,d\mu(y)\,d\mu(x).
    \end{multline}
    In addition, it follows from the heat kernel estimate~\eqref{E:assump_HKE} that, for $0 < \tau <1$ small enough,
    \begin{align*}
        &\int_X\int_{B(x,\tau)} p_{c\tau}(x,y)|f(x)-f(y)|\,d\mu(y)\,d\mu(x)\\
        & \le \frac{1}{\sigma_{\alpha_1/\beta_1,\alpha_2/\beta_2 }(c \tau)} \int_X\int_{B(x,\tau)}|f(x)-f(y)|\,d\mu(y)\,d\mu(x)\\
        & \le \frac{C\tau^{\alpha_1}}{\sigma_{\alpha_1/\beta_1,\alpha_2/\beta_2 }(c \tau)} \int_X\fint_{B(x,\tau)}|f(x)-f(y)|\,d\mu(y)\,d\mu(x) \\
        &\le \frac{C\tau^{\alpha_1} \sigma_{\alpha_1,\alpha_2 }( \tau)}{\sigma_{\alpha_1/\beta_1,\alpha_2/\beta_2 }(c \tau)} \|f\|_{1,\sigma_{\alpha_1/\beta_1,\alpha_2/\beta_2}}
        =C\tau^{\alpha_1(2-\frac{1}{\beta_1})}\|f\|_{1,\sigma_{\alpha_1/\beta_1,\alpha_2/\beta_2}}
    \end{align*}    
    and thus
    \begin{align}
        &\liminf_{\tau\to 0^+}\frac{1}{\tau}\sigma_{\frac{\kappa_1}{\beta_2},\frac{\kappa_2}{\beta_2}}(\tau)\int_X\int_{B(x,\tau)} p_{c\tau}(x,y)|f(x)-f(y)|\,d\mu(y)\,d\mu(x)\notag\\
        &\leq C\liminf_{\tau\to 0^+}\tau^{\frac{\kappa_1}{\beta_2}+\alpha_1(2-\frac{1}{\beta_1})-1}\|f\|_{1,\sigma_{\alpha_1/\beta_1,\alpha_2/\beta_2}}=0\label{E:pseudoPI_04}
    \end{align}
    since due to Assumption~\ref{A:HKE} and Assumption~\ref{A:VR_cond}, it holds that $\alpha_1(2-\frac{1}{\beta_1})-1\geq \frac{1}{2}$ and $\kappa_1>0>-\frac{1}{2}$.
    Finally, combining~\eqref{E:pseudoPI_02}-\eqref{E:pseudoPI_04}, it follows from~\eqref{E:pseudoPI_01} and Lemma~\ref{L:integral_sigma} that
    \begin{align*}
        &\bigg|\int_X(P_tf-f)g\,d\mu\bigg|\\
        &\leq C\|g\|_{L^\infty}\int_0^t\!\!\!\sigma_{\frac{\kappa_1}{\beta_1},\frac{\kappa_2}{\beta_2}}(s)\frac{ds}{s}\liminf_{\tau\to 0^+}\frac{\sigma_{\frac{\kappa_1}{\beta_1},\frac{\kappa_2}{\beta_2}}(\tau)}{\tau}\int_X\int_{X{\setminus}B(x,\tau)} p_{c\tau}(x,y)|f(x)-f(y)|\,d\mu(y)\,d\mu(x)\\
        &\leq C\|g\|_{L^\infty}\sigma_{1-\frac{\kappa_1}{\beta_1},1-\frac{\kappa_2}{\beta_2}}(t)\liminf_{\tau\to 0^+}\tau^{\frac{\kappa_1}{\beta_1}-1}\int_X\int_X p_{c\tau}(x,y)|f(x)-f(y)|\,d\mu(y)\,d\mu(x)\\
        &=C\|g\|_{L^\infty}\Psi_1(t)\liminf_{\tau\to 0^+}\frac{1}{\Psi_1(\tau)}\int_X\int_X p_{\tau}(x,y)|f(x)-f(y)|\,d\mu(y)\,d\mu(x)
    \end{align*}
    with $\Psi_1(t)=\sigma_{1-\frac{\kappa_1}{\beta_1},1-\frac{\kappa_2}{\beta_2}}(t)$.
\end{proof}

\section{\texorpdfstring{The case $p=1$}{[small]}} \label{S:case_p_1}
The aim of this section is to analyze in more detail the case $p=1$, corresponding to a theory of BV functions which in the fractal setting has not yet been treated with other existing constructions of $p$-energies, c.f.~\cite{KS24,MS25}. The main result in the section is that the weak Bakry-\'Emery estimate implies the $L^1$ weak monotonicity \eqref{A:PPI}, c.f. Theorem~\ref{T:weak_monotonicity}.

\medskip

We continue under the assumptions from Section~\ref{S:KS_char}, in particular we recall the functions $\Psi_1$ and $\Phi_1$ from the pseudo Poincar\'e inequality~\eqref{E:pseudoPI}.

\subsection{Weak monotonicity}
In the case $p=1$, Assumption~\ref{A:PPI} is often referred to as \emph{weak monotonicity}, c.f.~\cite{CGYZ24,KS24}. The main goal of this section is precisely to show that Assumption~\ref{A:PPI} actually holds true with $h_1=\Psi_1$ from the pseudo Poincar\'e inequality of Proposition~\ref{P:pseudoPI}.

\begin{theorem}\label{T:weak_monotonicity}
    Under the assumptions of Section~\ref{S:KS_char}, there exists $C>0$ such that 
    \begin{equation}\label{E:weak_monotonicity}
      \|f\|_{1,\Psi_1}\leq C\liminf_{t\to 0^+}\frac{1}{\Psi_1(t)}\int_XP_t(|f-f(y)|)(y)\,d\mu(y)
    \end{equation}
    for any $f\in \mathbf{B}^{1,\Psi_1}(X)$, where $\Psi_1(t)=\sigma_{1-\frac{\kappa_1}{\beta_1},1-\frac{\kappa_2}{\beta_2}}(t)$.
\end{theorem}

The proof consists of several lemmas and follows a strategy developed in~\cite[Section 4]{ABCRST3}. For the ease of notation, we start by defining the level sets
\begin{equation}\label{E:def_level_set}
    S_\lambda(f):=\{x\in X\colon f(x)>\lambda\}
\end{equation}
for any $f\in L^1(X,\mu)$ and $\lambda\in \mathbb{R}$, as well as the $h$-variation
that is naturally associated with the seminorm $\|f\|_{1,h}$, i.e.
\begin{equation*}
    {\rm\bf Var}_{h}(f):=\liminf_{t\to 0^+}\frac{1}{h(t)}\int_XP_t(|f-f(y)|)(y)\,d\mu(y).
\end{equation*}
With this notation, the weak monotonicity property~\eqref{E:weak_monotonicity} we want to prove reads
\begin{equation}\label{E:weak_monotonicity_w_Var}
    \|f\|_{1,\Psi_1}\leq C{\rm\bf Var}_{\Psi_1}(f)
\end{equation}
and Assumption~\ref{A:PPI} holds with $h_1(s)=\Psi_1(s)=\sigma_{1-\frac{\kappa_1}{\beta_1},1-\frac{\kappa_2}{\beta_2}}(s)$.

\medskip

We start by recalling two observations analogue to~\cite[Lemma 4.10, Lemma 4.11]{ABCRST3} that are still true in this generalized setting. They are proved in almost verbatim fashion and we thus refer the reader to~\cite{ABCRST3} for the details.
\begin{lemma}\label{L:prep4weak_mono_a}
    For any a.e. non-negative $f\in L^1(X,\mu)$ and $g\in L^\infty(X,\mu)$, and any $t>0$,
    \begin{multline*}
        \int_X\int_Xg(y)p_t(x,y)|f(x)-f(y)|\,d\mu(y)\,d\mu(x)\\
        \leq \int_0^\infty \Big(\|P_t(g\mathbf{1}_{S_\lambda (f)})-g\mathbf{1}_{S_\lambda(f)}\|_{L^1(X,\mu)}+\|P_t(g\mathbf{1}_{X{\setminus}S_\lambda(f)})-g\mathbf{1}_{X{\setminus}S_\lambda(f)}\|_{L^1(X,\mu)}\Big)\,d\lambda.
    \end{multline*}
    In particular, for $g\equiv 1$ it holds that
    \begin{equation}\label{E:prep4weak_mono_a_1}
      \int_X\int_Xp_t(x,y)|f(x)-f(y)|\,d\mu(y)\,d\mu(x)
        \leq  2\int_0^\infty\|P_t(\mathbf{1}_{S_\lambda(f)})-\mathbf{1}_{S_\lambda(f)}\|_{L^1(X,\mu)}d\lambda.
    \end{equation}
\end{lemma}
\begin{lemma}\label{L:prep4weak_mono_b}
    For any $h$ as in~\eqref{E:def_seminorm_ph} and a.e. non-negative $f\in L^1(X,\mu)$,
    \begin{equation}\label{E:prep4weak_mono_b}
        \int_0^\infty {\rm\bf Var}_{h}(\mathbf{1}_{S_\lambda(f)})\,d\lambda\leq {\rm\bf Var}_{h}(f).
    \end{equation}
\end{lemma}
With these in hand, we are now ready to prove the weak monotonicity assumption.
\begin{proof}[Proof of Theorem~\ref{T:weak_monotonicity}]
    Let $f_n:=(f+n)_{+}$.  
    Applying~\eqref{E:prep4weak_mono_a_1} to the function $f_n$ yields
    \[
    \int_X\int_X p_t(x,y)|f_n(x)-f_n(y)|\,d\mu(y)\,d\mu(x)\leq 2\int_0^\infty\|P_t\mathbf{1}_{S_\lambda(f_n)}-\mathbf{1}_{S_\lambda(f_n)}\|_{L^1(X,\mu)}d\lambda.
    \]
    By virtue of the Pseudo-Poincar\'e inequality from Proposition~\ref{P:pseudoPI} and Lemma~\ref{L:prep4weak_mono_b} it follows from the above that
    \begin{align*}
        \int_XP_t(|f_n-f_n(x)|)\,d\mu(x)&\leq 2\int_0^\infty \|P_t\mathbf{1}_{S_\lambda(f_n)}-\mathbf{1}_{S_\lambda(f_n)}\|_{L^1(X,\mu)}\,d\lambda\\
        &\leq C \Psi_1(t)\int_0^\infty{\rm\bf Var}_{\Psi_1}(\mathbf{1}_{S_\lambda(f_n)})\,d\lambda\\
        &\leq C\Psi_1(t){\rm\bf Var}_{\Psi_1}(f_n) \leq C \Psi_1(t){\rm\bf Var}_{\Psi_1}(f).
    \end{align*}
    
    Letting $n\to\infty$ we obtain
    \[
        \int_XP_t(|f-f(x)|)\,d\mu(x)\leq C\Psi_1(t){\rm\bf Var}_{\Psi_1}(f).
    \]
    Finally, dividing both sides of the inequality by $\Psi_1(t)$, and afterwards taking $\sup_{t>0}$ yields~\eqref{E:weak_monotonicity}.
\end{proof}

As a direct consequence of Theorem~\ref{T:weak_monotonicity} we obtain the following basic properties of ${\rm Var}_{\Psi_1}(f)$.

\begin{proposition}\label{P:props_Var}
    Let $f,g\in B^{1,\Psi_1}(X)$. Then,
    \begin{enumerate}[wide=0em,label={\rm (\roman*)}]
        \item ${\rm Var}_{\Psi_1}(f)=0$ implies $f\equiv 0$,
        \item there exists $C>0$ such that 
    \begin{equation*}
        {\rm Var}_{\Psi_1}(f+g)\leq C({\rm Var}_{\Psi_1}(f)+{\rm Var}_{\Psi_1}(g)),
    \end{equation*}
        \item if in addition $f,g\in L^\infty(X,\mu)$, then
        \begin{equation*}
            {\rm Var}_{\Psi_1}(fg)\leq C\|f\|_{L^\infty(X,\mu)}{\rm Var}_{\Psi_1}(g)+\|g\|_{L^\infty(X,\mu)}{\rm Var}_{\Psi_1}(g).
        \end{equation*}
    \end{enumerate}
\end{proposition}
\begin{proof}
    By virtue of Theorem~\ref{T:weak_monotonicity}, ${\rm Var}_{\Psi_1}(f)=0$ implies $\|f\|_{1,\Psi_1}= 0$ and hence $f\equiv 0$. This proves (i). Note next that 
    \[
    {\rm Var}_{\Psi_1}(h)\leq \|h\|_{1,\Psi_1}=\sup_{t>0}\frac{1}{\Psi_1(t)}\int_Xp_t(x,y)|f(x)-f(y)|\,d\mu(y).
    \]
    Setting $h=f+g$ and $h=fg$, respectively, and using
    \begin{align*}
     |(f+g)(x)-(f+g)(y)|&\leq |f(x)-f(y)|+|g(x)-g(y)|\\
     |fg(x)-fg(y)|&\leq |g(x)||f(x)-f(y)|+|g(y)||f(x)-f(y)|,
    \end{align*}
    the remaining properties (ii) and (iii) follow from Theorem~\ref{T:weak_monotonicity}.
\end{proof}

\subsection{Coarea formula and sets of finite perimeter}
The space $\mathbf{B}^{1,\Psi_1}(X)$ plays the role of bounded variation functions, ${\rm BV}(X)$ and the variance ${\rm Var}_{\Psi_1}$ now allows to define a notion of perimeter measure for measurable sets. In that analogy, Theorem~\ref{T:coarea_formula} provides the corresponding co-area formula, whose proof is analogue to that of~\cite[Theorem 4.15]{ABCRST3}. 

\begin{definition}
    The perimeter of a $E\subset X$ is defined as
    \begin{equation}\label{E:def_perimeter}
        {\rm Per}_{\Psi_1}(E):={\rm Var}_{\Psi_1}(\mathbf{1}_E).
    \end{equation}
\end{definition}

The details of the next proof are included for completeness. To that end, recall the notation $S_\lambda(f)$ for level sets introduced in~\eqref{E:def_level_set}.

\begin{theorem}\label{T:coarea_formula}
    For any a.e. non-negative function $f\in L^1(X,\mu)$,
    \begin{equation}\label{E:coarea_formula}
        {\rm Var}_{\Psi_1}(f)\simeq \int_0^\infty{\rm Var}_{\Psi_1}(\mathbf{1}_{S_\lambda(f)})\,d\lambda.
    \end{equation}
\end{theorem}

\begin{proof}
    The left inequality in~\eqref{E:coarea_formula} is a direct consequence from Lemma~\ref{L:prep4weak_mono_b}. To prove the inequality on the right-hand side, applying~\eqref{E:prep4weak_mono_a_1} and the definition of the seminorm $\|\cdots\|_{1,\Psi_1}$ we obtain
    \begin{align*}
        \int_X P_t(f-f(x))\,d\mu(x)
        &\leq 2\int_0^\infty\int_X|(P_t-I)\mathbf{1}_{S_\lambda(f)}|\,d\mu\,d\lambda\\
        &\leq \Psi_1(t)\int_0^\infty\|\mathbf{1}_{S_\lambda(f)}\|_{1,\Psi_1}d\lambda.
    \end{align*}
    Dividing now both sides of the inequality by $\Psi_1(t)$, applying Theorem~\ref{T:weak_monotonicity}, and eventually taking $\liminf_{t\to 0^+}$, the latter implies
    \begin{align*}
        {\rm Var}_{\Psi_1}(f)
        &=\liminf_{t\to 0^+}\frac{1}{\Psi_1(t)}\int_X P_t(f-f(x))\,d\mu(x)\\
        &\leq C \liminf_{\tau\to\infty}\frac{1}{\Psi_1(\tau)}\int_XP_\tau(\mathbf{1}_{S_\lambda(f)}-\mathbf{1}_{S_\lambda(f)}(x)|)\,d\mu(x)
        ={\rm Var}_{\Psi_1}(\mathbf{1}_{S_\lambda(f)})
    \end{align*}
    as required.
\end{proof}

\subsection{Orlicz-Sobolev embedding}
In this section we turn our attention toward the Sobolev embedding in this setting. Orlicz norms become handy at this point since they are able to capture the presence of different scaling factors. In addition, the isoperimetric-type inequality that follows from the embedding, see~\eqref{E:Orlicz_isoperimetric}, provides a first non-trivial lower bound for the perimeter of sets, which is sensitive to the size of the measure of the set.

\begin{theorem}\label{T:Sobolev_embedding}
    Let the assumptions of Section~\ref{S:KS_char} and ${\rm wBE}(\kappa_1,\kappa_2)$ be satisfied with $0<\kappa_i<\alpha_i-\beta_i$, $i=1,2$. Further, let 
    \begin{equation*}
      \phi_1(s):=\sigma_{\frac{\alpha_1}{\alpha_1-\beta_1+\kappa_1},\frac{\alpha_2}{\alpha_2-\beta_2+\kappa_2}}(s)\quad\text{and}\quad \Phi_1(s):=\sigma_{1+\frac{\kappa_1-\beta_1}{\alpha_1},1+\frac{\kappa_2-\beta_2}{\alpha_2}}(s).
    \end{equation*}
    Then, there exists $C>0$ such that
    \begin{equation}\label{E:Sobolev_embedding}
        \|f\|_{\phi_1}\leq C{\rm Var}_{\Psi_1}(f)
    \end{equation}
    for any $f\in\mathbf{B}^{1,\Psi_1}(X)$. In particular, for any set $E\subset X$ of finite $\Psi_1$ perimeter,
    \begin{equation}\label{E:Orlicz_isoperimetric}
        \Phi_1(\mu(E))\leq C {\rm Per}_{\Psi_1}(E).
    \end{equation}
\end{theorem}

\begin{proof}
To prove~\eqref{E:Sobolev_embedding}, it suffices to find the function $\phi$ verifying the hypothesis from Theorem~\ref{T:pol}. The heat kernel estimate~\eqref{E:assump_HKE} implies \eqref{eq:subGauss-upper4} with $V(t)=\sigma_{\frac{\alpha_1}{\beta_1},\frac{\alpha_2}{\beta_2}}(t)$, while the monotonicity property from Theorem~\ref{T:weak_monotonicity} gives $h_1(t)=\Psi_1(t)=\sigma_{1-\frac{\kappa_1}{\beta_1},1-\frac{\kappa_2}{\beta_2}}(t)$. By virtue of Proposition~\ref{P:props_sigma},
\begin{align*}
\sigma_{1-\frac{\kappa_1}{\beta_1},1-\frac{\kappa_2}{\beta_2}}(\sigma_{\frac{\alpha_1}{\beta_1},\frac{\alpha_2}{\beta_2}}^{-1}(s))
&=\sigma_{1-\frac{\kappa_1}{\beta_1},1-\frac{\kappa_2}{\beta_2}}(\sigma_{\frac{\beta_1}{\alpha_1},\frac{\beta_2}{\alpha_2}}(s))
=\sigma_{\frac{\beta_1-\kappa_1}{\alpha_1},\frac{\beta_2-\kappa_2}{\alpha_2}}(s)\\
&=s\sigma_{\frac{\beta_1-\kappa_1}{\alpha_1}-1,\frac{\beta_2-\kappa_2}{\alpha_2}-1}(s)
=s\sigma_{1-\frac{\beta_1-\kappa_1}{\alpha_1},1-\frac{\beta_2-\kappa_2}{\alpha_2}}(1/s)\\
&=s\sigma_{\frac{\alpha_1}{\alpha_1-\beta_1+\kappa_1},\frac{\alpha_2}{\alpha_2-\beta_2+\kappa_2}}^{-1}(1/s)
\end{align*}
hence the hypothesis from Theorem~\ref{T:pol} are satisfied with $\phi(s)=\phi_1(s)$. 

\medskip

To obtain~\eqref{E:Orlicz_isoperimetric}, we first compute the Young conjugate of $\phi_1$, which is bounded above and below by $\psi_1=\sigma_{\frac{\alpha_1}{\beta_1-\kappa_1},\frac{\alpha_2}{\beta_2-\kappa_2}}$. By virtue of~\eqref{E:Norm_1E}, it now follows that $\|\mathbf{1}_E\|_{\phi_1}$ is bounded above and below by
\begin{align*}
&\mu(E)\sigma_{\frac{\alpha_1}{\beta_1-\kappa_1},\frac{\alpha_2}{\beta_2-\kappa_2}}^{-1}(1/\mu(E))
=\mu(E)\sigma_{\frac{\beta_1-\kappa_1}{\alpha_1},\frac{\beta_2-\kappa_2}{\alpha_2}}(1/\mu(E))\\
&=\mu(E)\sigma_{\frac{\kappa_1-\beta_1}{\alpha_1},\frac{\kappa_2-\beta_2}{\alpha_2}}(\mu(E))
=\sigma_{1+\frac{\kappa_1-\beta_1}{\alpha_1},1+\frac{\kappa_2-\beta_2}{\alpha_2}}(\mu(E)).
\end{align*}
Setting $f=\mathbf{1}_E$ in~\eqref{E:Sobolev_embedding} finally yields~\eqref{E:Orlicz_isoperimetric}.
\end{proof}

\begin{remark}\label{R:sanity_check_Orlicz-Sobolev}
    In the case of standard sub-Gaussian heat kernel estimates and the standard volume doubling property, $\alpha_1=\alpha_2=d_H$ and $\beta_1=\beta_2=d_W$. Under the weak Bakry-\'Emery condition ${\rm wBE}(\kappa)$, it holds that $\frac{\kappa_1}{\beta_1}=\frac{\kappa_2}{\beta_2}=\frac{\kappa}{d_W}$ and hence
    \[
    \phi_1(s)=s^{\frac{d_H}{d_H-d_W+\kappa}}\qquad\text{and}\qquad\psi_1(s)=s^{\frac{d_H}{d_W-\kappa}}.
    \]
    Thus, the embedding~\eqref{E:Sobolev_embedding} reads 
    \[
    \|f\|_{L^{\frac{d_H}{d_H-d_W+\kappa}}}\leq C{\rm\bf Var}_{\Psi_1}(f)
    \]
    and the isoperimetric inequality~\eqref{E:Orlicz_isoperimetric} is
    \[
    \mu(E)^{1+\frac{\kappa-d_W}{d_H}}\leq C {\rm Per}_{\Psi_1}(E),
    \]
    which is equivalent to~\cite[Theorem 4.18]{ABCRST3}.
\end{remark}

We finish this section by shortly addressing the degenerate case when $\kappa_i=\alpha_i-\beta_i$, $i=1,2$. Here one can follow verbatim the proof of~\cite[Proposition 4.20]{ABCRST3} to estimate the variance ${\rm\bf Var}_{\Psi_1}(f)$ from below by the oscillations of $f$
\[
{\rm \bf Osc}(f):={\rm ess}\!-\!\sup_{x,y\in X}|f(x)-f(y)|.
\]

\begin{proposition}\label{P:Var_vs_osc}
    Let the assumptions of Section~\ref{S:KS_char} and ${\rm wBE}(\kappa_1,\kappa_2)$ be satisfied with $0<\kappa_i=\alpha_i-\beta_i$, $i=1,2$. Then, there exists $C>0$ such that 
    \[
    {\rm \bf Osc}(f)\leq C {\rm\bf Var}_{\Psi_1}(f)
    \]
    for any $f\in \mathbf{B}^{1,\Psi_1}(X)$ and $\mathbf{B}^{1,\Psi_1}(X)\subset L^\infty (X,\mu)$.
\end{proposition}

\section*{Appendix}
In this section we record several technical lemmata that are repeatedly applied in Section~\ref{S:KS_char}. Recall the definition of the function 
\[
\sigma_{\alpha,\beta} (r)=r^\alpha 1_{[0,1]}(r) +r^\beta 1_{(1,+\infty)}(r)
\]
for any $\alpha,\beta>0$.

\begin{lemma}\label{estimate comp 1}
Let $C>0$,  $\alpha_1, \alpha_2 > 0$ and $\beta_1,\beta_2>1$. There exists a constant $c>0$ such that for every $t >0$
\begin{align*}
 \int_{0}^{+\infty}\sigma_{\alpha_1,\alpha_2}(u)\exp\Big(\!\!-C t \sigma_{\beta_2/(\beta_2-1),\beta_1/(\beta_1-1) } \left( \frac{u}{t}\right)\Big) \frac{du}{u} \le c \sigma_{\alpha_1/\beta_1,\alpha_2/\beta_2 }(t)
\end{align*}
\end{lemma}

\begin{proof}
We first assume that $0 \le t \le 1$ and split the integral $\int_0^{+\infty}$ into $\int_0^t +\int_t^1 +\int_1^{+\infty}$. The first integral is given by
\begin{align*}
A(t):=\int_{0}^{t} u^{\alpha_1} \exp\Big(\!\!-C \Big(\frac{u^{\beta_2}}{t}\Big)^{\frac{1}{\beta_2-1}}\Big) \frac{du}{u}
\end{align*}
Using the change of variable $u=t^{1/\beta_2} s$, we get
\begin{align*}
A(t)=t^{\frac{\alpha_1}{\beta_2}}\int_0^{t^{1-\frac{1}{\beta_2}}} s^{\alpha_1-1} \exp\Big(\!\!-C s^{\frac{\beta_2}{\beta_2-1}}\Big) ds  \le  t^{\frac{\alpha_1}{\beta_2}}\int_0^{t^{1-\frac{1}{\beta_2}}} s^{\alpha_1-1}  ds=\frac{t^{\alpha_1}}{\alpha_1}.
\end{align*}
The second integral is given by
\begin{align*}
B(t):=\int_t^1 u^{\alpha_1} \exp\Big(\!\!-C \Big(\frac{u^{\beta_1}}{t}\Big)^{\frac{1}{\beta_1-1}}\Big) \frac{du}{u}.
\end{align*}
It can be estimated as follows
\begin{align*}
B(t) \le \int_0^{+\infty} u^{\alpha_1} \exp\Big(\!\!-C \Big(\frac{u^{\beta_1}}{t}\Big)^{\frac{1}{\beta_1-1}}\Big) \frac{du}{u} = t^{\frac{\alpha_1}{\beta_1}}  \int_0^{+\infty} u^{\alpha_1} \exp\Big(\!\!-C u^{\frac{\beta_1}{\beta_1-1}}\Big) \frac{du}{u}.
\end{align*}
For the third integral, we have
\[
C(t):=\int_1^{+\infty} u^{\alpha_2} \exp\Big(\!\!-C \Big(\frac{u^{\beta_1}}{t}\Big)^{\frac{1}{\beta_1-1}}\Big)\frac{du}{u},
\]
which can be estimated as follows
\[
C(t) \le \exp\Big(\!\!-\frac{C}{2t^{\frac{1}{\beta_1-1}}} \Big) \int_0^{+\infty} u^{\alpha_2} \exp\Big(\!\!-\frac{C}{2} \Big(\frac{u^{\beta_1}}{t}\Big)^{\frac{1}{\beta_1-1}}\Big)\frac{du}{u}\le c t^{\frac{\alpha_2}{\beta_1}} \exp\Big(\!\!-\frac{C}{2t^{\frac{1}{\beta_1-1}}} \Big).
\]
Putting the estimates together yields that for $0 \le t \le 1$
\[
 \int_{0}^{+\infty}\sigma_{\alpha_1,\alpha_2}(u)\exp\Big(\!\!-C t \sigma_{\beta_2/(\beta_2-1),\beta_1/(\beta_1-1) } \left( \frac{u}{t}\right)\Big) \frac{du}{u} \le A(t) +B(t) +C(t) \le C  t^{\frac{\alpha_1}{\beta_1}}.
\]
For the case $t \ge 1$, we split the integral $\int_0^{+\infty}$ as $\int_0^1 +\int_1^t +\int_t^{+\infty}$. The first integral is
\[
A(t):=\int_0^1 u^{\alpha_1} \exp\Big(\!\!-C \Big(\frac{u^{\beta_2}}{t}\Big)^{\frac{1}{\beta_2-1}}\Big) \frac{du}{u} \le \int_0^1 u^{\alpha_1-1} du.
\]
The second integral is 
\[
B(t):=\int_1^t  u^{\alpha_2} \exp\Big(\!\!-C \Big(\frac{u^{\beta_2}}{t}\Big)^{\frac{1}{\beta_2-1}}\Big) \frac{du}{u}.
\]
We estimate $B(t)$ as
\[
B(t) \le \int_0^{+\infty}  u^{\alpha_2} \exp\Big(\!\!-C \Big(\frac{u^{\beta_2}}{t}\Big)^{\frac{1}{\beta_2-1}}\Big) \frac{du}{u} \le C t^{\frac{\alpha_2}{\beta_2}}.
\]
Finally, the last integral is given by 
\[
C(t)=\int_t^{+\infty} u^{\alpha_2} \exp\Big(\!\!-C \Big(\frac{u^{\beta_1}}{t}\Big)^{\frac{1}{\beta_1-1}}\Big)\frac{du}{u}.
\]
Using the change of variable $u=t^{1/\beta_1} s$, we get
\begin{align*}
C(t)&=t^{\frac{\alpha_2}{\beta_1}}\int_{t^{1-\frac{1}{\beta_1}}}^{+\infty} s^{\alpha_2-1} \exp\Big(\!\!-C s^{\frac{\beta_1}{\beta_1-1}}\Big) ds \\
 &\le t^{\frac{\alpha_2}{\beta_1}}\int_{t^{1-\frac{1}{\beta_1}}}^{+\infty} s^{\alpha_2-1} \exp\Big(\!\!-C s^{\frac{\beta_1}{\beta_1-1}}\Big) ds \\
 & \le t^{\frac{\alpha_2}{\beta_1}} e^{-\frac{C}{2} t}\int_{0}^{+\infty} s^{\alpha_2-1} \exp\Big(\!\!-\frac{C}{2} s^{\frac{\beta_1}{\beta_1-1}}\Big) ds. 
\end{align*}
Putting the estimates together yields that for $t \ge 1$
\[
 \int_{0}^{+\infty}\sigma_{\alpha_1,\alpha_2}(u)\exp\Big(\!\!-C t \sigma_{\beta_2/(\beta_2-1),\beta_1/(\beta_1-1) } \left( \frac{u}{t}\right)\Big) \frac{du}{u} \le A(t) +B(t) +C(t) \le C  t^{\frac{\alpha_2}{\beta_2}}.
\]

\end{proof}

\begin{lemma}
Let $C>0$ $\alpha_1, \alpha_2 > 0$, $\beta_1,\beta_2>1$ and $\nu_1,\nu_2>0$. There exists a constant $c>0$ such that 
\begin{multline*}
 \sum_{k=1}^{+\infty} \exp\Big(\!\!-C t \sigma_{\beta_2/(\beta_2-1),\beta_1/(\beta_1-1) } \left( \frac{2^{-k}r}{t}\right)\Big) \sigma_{\nu_1\beta_1,\nu_2\beta_2}(2^{1-k}r)^p\sigma_{\alpha_1,\alpha_2 }(2^{1-k}r) \\
  \le  c \sigma_{\alpha_1/\beta_1,\alpha_2/\beta_2 }(t) \sigma_{\nu_1,\nu_2}(t)^p
\end{multline*}
for any $r,t >0$.
\end{lemma}

\begin{proof}
We have
\begin{align*}
 & \sum_{k=1}^{+\infty} \exp\Big(\!\!-C t \sigma_{\beta_2/(\beta_2-1),\beta_1/(\beta_1-1) } \left( \frac{2^{-k}r}{t}\right)\Big) \sigma_{\nu_1\beta_1,\nu_2\beta_2}(2^{1-k}r)^p\sigma_{\alpha_1,\alpha_2 }(2^{1-k}r) \\
 \le &c \int_0^r \sigma_{\nu_1\beta_1,\nu_2\beta_2}(u)^p \sigma_{\alpha_1,\alpha_2}(u)\exp\Big(\!\!-C t \sigma_{\beta_2/(\beta_2-1),\beta_1/(\beta_1-1) } \left( \frac{u}{t}\right)\Big) \frac{du}{u} \\
 \le & c \int_0^{+\infty} \sigma_{\nu_1\beta_1,\nu_2\beta_2}(u)^p \sigma_{\alpha_1,\alpha_2}(u)\exp\Big(\!\!-C t \sigma_{\beta_2/(\beta_2-1),\beta_1/(\beta_1-1) } \left( \frac{u}{t}\right)\Big) \frac{du}{u} \\
 \le &  c \sigma_{\alpha_1/\beta_1,\alpha_2/\beta_2 }(t) \sigma_{\nu_1,\nu_2}(t)^p,
\end{align*}
where the last estimate is obtained as before (use Lemma \ref{estimate comp 1} with $\sigma_{p\nu_1\beta_1+\alpha_1,p\nu_2\beta_2+\alpha_2} $ instead of $\sigma_{\alpha_1,\alpha_2}$) .
\end{proof}

\begin{lemma}\label{L:integral_sigma}
Let $\nu_1,\nu_2>1$. There exists $C>0$ such that for any $t> 0$ 
\begin{equation*}
    \int_0^t\frac{1}{\sigma_{\nu_1,\nu_2}(u)}\,du\leq C\sigma_{1-\nu_1,1-\nu_2}(t).
\end{equation*}
\end{lemma}
\begin{proof}
    When $0< t\leq 1$, then
    \begin{equation*}
        \int_0^t\frac{1}{\sigma_{\nu_1,\nu_2}(u)}\,du
        =\int_0^tu^{-\nu_1}\,du= \frac{1}{1-\nu_1}t^{1-\nu_1}
    \end{equation*}
    whereas if $t>0$, then
    \begin{equation*}
        \int_0^t\frac{1}{\sigma_{\nu_1,\nu_2}(u)}\,du
        =\int_0^1u^{-\nu_1}\,du+\int_1^tu^{-\nu_2}du
        =\frac{1}{1-\nu_1}+\frac{1}{1-\nu_2}(t^{1-\nu_2}-1)\leq Ct^{1-\nu_2}.
    \end{equation*}
\end{proof}

\newpage
\bibliographystyle{amsplain}
\bibliography{vBesov_refs}

\vspace{5pt}
\noindent
\begin{minipage}{\textwidth}
    \small
    \textbf{Patricia Alonso Ruiz:} \\
    Department of Mathematics, Friedrich Schiller University Jena\\
    Email: patricia.alonso.ruiz@uni-jena.de
\end{minipage}

\vspace{10pt} 

\noindent
\begin{minipage}{\textwidth}
    \small
    \textbf{Fabrice Baudoin:} \\
    Department of Mathematics, Aarhus University \\
    Email: fbaudoin@math.au.dk
\end{minipage}

\end{document}